\documentclass[11pt, a4paper]{amsart}
\usepackage{amssymb}
\usepackage{palatino}
\input amssym.def

\usepackage{amsmath, amsfonts, verbatim}
\usepackage[hyperfootnotes=false, colorlinks, linkcolor={blue}, citecolor={magenta}, filecolor={blue}, urlcolor={blue}]{hyperref}
\usepackage{amssymb, color}
\usepackage{amscd}
\usepackage[mathscr]{eucal}
\usepackage{palatino}
\usepackage[left=2.5cm,right=2.5cm,top=2.5cm,bottom=2.5cm]{geometry}

\newfont{\cyrr}{wncyr10}

\renewcommand{\H}{\mathbb H}
\renewcommand{\Im}{\mathrm{Im}}
\renewcommand{\Re}{\mathrm{Re}}
\newcommand{\GL}{{\rm GL}}

\newcommand{\SO}{{\rm SO}}

\newcommand{\GSp}{{\rm GSp}}
\newcommand{\Sp}{{\rm Sp}}
\newcommand{\PGSp}{{\rm PGSp}}

\newcommand{\sym}{{\rm sym}}

\newcommand{\Cl}{{\rm Cl}}

\newcommand{\bs}{\backslash}
\newcommand{\eps}{\epsilon}
\newcommand{\disc}{{\rm disc}}
\newcommand{\G}{{\Gamma}}

\newcommand{\h}{{\mathcal H}}
\newcommand{\n}{{\mathcal N}}
\newcommand{\N}{{\mathbb N}}
\newcommand{\Z}{{\mathbb Z}}
\newcommand{\Q}{{\mathbb Q}}
\newcommand{\R}{{\mathbb R}}
\newcommand{\C}{{\mathbb C}}

\newcommand{\A}{{\mathbb A}}
\newcommand{\SL}{{\rm SL}}

\renewcommand{\mod}{{\, \rm mod \, }}
\newcommand{\mat}[4]{{\setlength{\arraycolsep}{0.5mm}\left[
\begin{smallmatrix}#1&#2\\#3&#4\end{smallmatrix}\right]}}

\newcommand{\corref}[1]{Corollary~\ref{#1}}
\newcommand{\thmref}[1]{Theorem~\ref{#1}}
\newcommand{\propref}[1]{Proposition~\ref{#1}}
\newcommand{\lemref}[1]{Lemma~\ref{#1}}

\newtheorem{thm}{Theorem}[subsection]

\newtheorem{lem}[thm]{Lemma}
\newtheorem{cor}[thm]{Corollary}
\newtheorem{prop}[thm]{Proposition}
\newtheorem{rmk}[thm]{Remark}

\newtheorem{conj}[thm]{Conjecture}

\begin{document}

\title[Fourier coefficients and Hecke eigenvalues]
{On Fourier coefficients and Hecke eigenvalues of Siegel cusp forms of
degree $2$}
\author{Biplab Paul and Abhishek Saha}

\address[Biplab Paul]
{School of Mathematical Sciences, Queen Mary University of London,
Mile End Road, London, E1 4NS, UK.
\newline
Current Address: Chennai Mathematical Institute, H1, SIPCOT IT Park,
Siruseri, Kelambakkam 603103, Tamil Nadu, India.}
\email{bpaul@cmi.ac.in, paulbiplab0@gmail.com}

\address[Abhishek Saha]
{School of Mathematical Sciences, Queen Mary University of London,
Mile End Road, London, E1 4NS, UK.}
\email{abhishek.saha@qmul.ac.uk}

\date{\today}

\begin{abstract}
We investigate some key analytic properties of Fourier coefficients and Hecke
eigenvalues attached to scalar-valued Siegel cusp forms $F$ of degree 2, weight $k$ and level $N$.
First, assuming that $F$ is a Hecke eigenform that is not of Saito-Kurokawa
type, we prove an improved bound in the $k$-aspect for the smallest
prime at which its Hecke eigenvalue is negative. Secondly, we show
that there are infinitely many sign changes among the Hecke eigenvalues of $F$
at primes lying in an arithmetic progression. Third, we show that there are
infinitely many positive as well as infinitely many negative Fourier
coefficients in any ``radial" sequence comprising of
prime multiples of a fixed fundamental matrix. Finally we consider
the case when $F$ is of Saito--Kurokawa type, and in this case we
prove the (essentially sharp) bound
$| a(T) |
~\ll_{F, \epsilon}~
\big( \det T \big)^{\frac{k-1}{2}+\epsilon}$
for the Fourier coefficients of $F$ whenever $\gcd(4 \det(T), N)$ is squarefree,
confirming a conjecture made (in the case $N=1$)
by Das and Kohnen. %For the last two results we do not assume that $F$ is a Hecke eigenform. In contrast to most related previous works, our results are valid for Siegel cusp forms with respect to congruence subgroups of general level $N$.
\end{abstract}

\maketitle

\section{Introduction}
For positive integers $n, k$, and $N$, let $S_k(\Gamma^{(n)}_0(N))$
be the space of holomorphic cusp forms of weight $k$ for the Siegel-type
congruence subgroup $\Gamma^{(n)}_0(N) \subseteq \mathrm{Sp}_{2n}(\Z)$
of level $N$ and degree $n$. Any $F \in S_k(\Gamma^{(n)}_0(N))$
has a Fourier expansion
$$
F(Z)
~=~
\sum_{T \in \n_n} a(T)e^{2\pi i \rm{tr}(TZ)},
$$
where
$$
\n_n ~:=~ \big\{ T = (t_{ij})_{n \times n} ~|~ 2t_{ij}, t_{ii} \in \Z, ~T^t = T,  ~T > 0 \big\}.
$$
Here and throughout the article, the symbol $T > 0$ means $T$ is
positive definite. If $F$ is a Hecke eigenform for the Hecke operators $T(m)$ with $\gcd(m,N)=1$, we denote its Hecke
eigenvalues by $\eta(m)$. Both the Fourier coefficients $a(T)$ and
Hecke eigenvalues $\eta(m)$ are key objects and they, separately
and together, play an important role in understanding the Hecke
eigenforms and hence all cusp forms.

\begin{comment}
In this article, we investigate the behaviour of the signs
of $\eta(p)$ when $p$ varies over primes and certain arithmetically
interesting subsequence of primes, that of radial Fourier coefficients
$a(mT)$ where $m \in \N$ varies and $T$ is fixed, and their
inter-relations. When $F$ is in (an appropriate generalization to
the case of level $N$ of) the Maass subspace, we also prove an
upper bound for $|a(T)|$ whose exponent is best possible.
\end{comment}

It is well-known that when $n = 1$, i.e., when $F$ is an elliptic Hecke eigenform,
the Fourier coefficients and Hecke eigenvalues coincide up to a certain normalization.
In this case, one can exploit the multiplicative properties of $\eta(m)$, the Ramanujan
bound, and the Hecke relations to study sign change questions for $\eta(m)$.
Some interesting results in this set-up can be found in
\cite{KS06, KLSW10, KM12, MR15, GKP19, GMP21}
and the references therein.

In the case of Siegel cusp forms of degree $n=2$, Hecke eigenvalues and
Fourier coefficients of a Hecke eigenform are not related in any simple
way. Further, it is not clear how to make sense of multiplicative properties
of Fourier coefficients (which are indexed by matrices). Indeed, Fourier
coefficients are closely related to central values of degree 8 $L$-functions
and are more mysterious than the
Hecke eigenvalues \cite{DPSS20, JLS21+, FurMori}. For example, even
though an analogue of Ramanujan's conjecture is known for the Hecke
eigenvalues corresponding to an eigenform of degree $n=2$, the analogous
conjecture for Fourier coefficients (which is famously known as
Resnikoff--Salda\~na conjecture \cite{RS74})
is not known yet. For recent progress towards the Resnikoff--Salda\~na conjecture,
see \cite{JLS21+}. In this paper we investigate the following questions for $n=2$:
\begin{itemize}
\item
Bounds for the first prime for which the Hecke eigenvalue is negative
(analogue of least quadratic non-residue).

\item
Sign changes of Hecke eigenvalues at primes lying in any fixed arithmetic progression.

\item
Behaviour of signs of Fourier coefficients of a non-zero
cusp form (not necessarily a Hecke eigenform) along a sequence of
prime multiples of a fixed matrix.

\item
Bounds for the Fourier coefficients of a cusp form lying in the
generalized Maass subspace.
\end{itemize}
Throughout this paper, we allow the level $N$ to be any integer,
in contrast to many previous related works which were restricted to $N=1$
or $N$ squarefree.

\subsection{Main Results on signs of Hecke eigenvalues}
In the case $N=1$, the space $S_k(\Sp_4(\Z)) = S_k(\G_0^{(2)}(1))$ of Siegel cusp forms of weight $k$ and full level
has a natural subspace denoted
$S_k^*(\Sp_4(\Z))$ and called the \emph{Maass subspace}, which is the span
of the \emph{Saito-Kurokawa lifts} of full level. The Saito-Kurokawa lifts of full level can be
explicitly constructed from classical half-integral weight forms via the theory
of Jacobi forms  \cite[\S6]{EZ85}; they may also be viewed as lifts of classical
integral weight forms thanks to the Shimura correspondence between
half integral weight and integral weight forms. It is known \cite{SB99, GPS18}
that all the Hecke eigenvalues of a Hecke eigenform lying in the
Maass subspace are positive. Indeed, for a Hecke eigenform $F$ lying in the
Maass subspace, there exists a classical Hecke eigenform
$f \in S_{2k-2}(\SL_2(\Z))$ such that
\begin{equation}\label{SK}
\lambda_F(p) = \lambda_f(p) + p^{1/2} + p^{-1/2}
\end{equation}
where $\lambda_F(p)$ (resp. $\lambda_f(p)$) are the normalized Hecke
eigenvalues of $F$ (resp., $f$). The relation \eqref{SK} clearly shows the
positivity of the quantities $\lambda_F(p)$ (since we know that $|\lambda_f(p)| \le 2$).
A simple argument extends this positivity result to $\lambda_F(n)$ for
all $n$. Therefore, the topic of \emph{signs of Hecke eigenvalues} of Hecke eigenforms in $S_k(\Sp_4(\Z))$ is interesting only for
forms that lie in the orthogonal complement\footnote{It is easy to see that if a Hecke eigenform does not lie in $S_k^*(\Sp_4(\Z))$, then it must be contained in the orthogonal complement of $S_k^*(\Sp_4(\Z))$.} of $S_k^*(\Sp_4(\Z))$.

For general $N$, we define a natural generalization of the Maass subspace
which we call the \emph{generalized Maass subspace} and denote by $S_k^*(\G_0^{(2)}(N))$. The space $S_k^*(\G_0^{(2)}(N))$
is defined to be the span of certain Hecke eigenforms $F$ that we say are of \emph{classical Saito-Kurokawa type}.
These eigenforms  of classical Saito-Kurokawa type are defined by us in the
language of Arthur packets (see Section \ref{s:arthur}) and turn out to have the property
that their Hecke eigenvalues at primes $p \nmid N$ are all positive. We remark that there is also a separate
explicit construction of Saito-Kurokawa lifts for $\G_0^{(2)}(N)$ for general $N$
via the route of Jacobi forms, due to Ibukiyama \cite{TI12} (see also Agarwal--Brown \cite{AB15}).
 To the best of our knowledge it is not known whether the  Saito-Kurokawa lifts of trivial character constructed by Ibukiyama are sufficient to fully capture\footnote{It can be shown that if $N$ is \emph{squarefree} then the level $N$ newforms that are of classical Saito-Kurokawa type in our sense are exactly the same as the Saito--Kurokawa lifts of newforms of level $N$ constructed via Jacobi forms in works of Ibukiyama and Agarwal--Brown; so the main question is what happens for non-squarefree levels.} our generalized Maass space $S_k^*(\G_0^{(2)}(N))$ as defined in Section \ref{s:arthur}.

As mentioned above, for eigenforms
of classical Saito-Kurokawa type the Hecke eigenvalues at primes not dividing the level are all positive. Consider an eigenform that is \emph{not}
of classical Saito-Kurokawa type. Our first result provides a bound in the $k$-aspect for
the smallest prime at which the Hecke eigenvalue is negative.

\begin{thm}\label{First-sign}
Let $k$ and $N$ be positive integers. Let $F \in S_k(\G_0^{(2)}(N))$ be a Hecke eigenform with Hecke
eigenvalues $\eta(m)$ for all positive integers $m$ satisfying $\gcd(m, N)=1$. Assume that $F$ does not lie in $S_k^*(\G_0^{(2)}(N))$. If $k=2$ and $F$ is not attached to a CAP representation, then assume further that $F$ satisfies the Ramanujan conjecture.\footnote{The Ramanujan conjecture is known for all weight $k\ge 3$. Readers unfamiliar with the definition of a CAP representation or the Ramanujan conjecture should refer to Sections \ref{s:autprelim} and \ref{s:arthur} .}
Then, for any $\epsilon > 0$, there exists a prime $p$ with $\gcd(p, N)=1$
and
$$
p ~\ll_{N,\epsilon}~ k^{5+\epsilon}
$$
such that $\eta(p) < 0$. The implicit constant depends
polynomially\footnote{Our method can be used to make this polynomial explicit by carefully going through
the proof of Lemma \ref{l:constants}.} on $N$.
\end{thm}

%Note that for the Hecke eigenvalues of a normalized (elliptic) Hecke eigenform, i.e., when degree $n =1$, the bounds for the first sign change of Hecke eigenvalues over primes was proved by Jin \cite[Theorem 3.1]{SJ19}.

We also prove that for eigenforms not
of classical Saito-Kurokawa type, the associated Hecke eigenvalues $\eta(p)$ have infinitely many sign changes as $p$ traverses a typical arithmetic progression. To
 the best of our knowledge, this is the first result of this kind for Siegel cusp forms of degree 2.

\begin{thm}\label{split-inert}
Let $F$ be as in Theorem \ref{First-sign}.  Then, for any
positive integers $a, M$ with $\gcd(aN, M)=1$, the sequence
$( \eta(p) )$ with $p$ varying over primes congruent to $a \mod M$ changes sign
infinitely often.
\end{thm}

Let us discuss briefly  previous results related  to the
above theorems. The investigation of first negative Hecke
eigenvalue of Siegel cusp forms  was initiated by Kohnen and Sengupta \cite{KS07}
(for the case $N=1$).
Exploring relations among Hecke eigenvalues at
prime power indices they proved the existence of $n \ll k^2\log^{20} k$
such that $\eta(n) < 0$. Note that in their result $n$ can be any integer, in contrast to our
Theorem \ref{First-sign} where $n$ is a prime. The study of signs of Hecke eigenvalues at all integers
has a different flavour (and is usually easier) from studying
sign-changes over the primes. We remark here that the Hecke algebra
is generated not just by $T(p)$ but by $T(p)$ and
$T(p^2)$, so going down from all integers to primes is non-trivial. Assuming the Generalized Lindel\"of hypothesis
one can improve the bound obtained by Kohnen and Sengupta to $n \ll_\epsilon k^\epsilon$ for any
$\epsilon > 0$. The result of Kohnen and Sengupta was generalized
by Brown \cite{JB10} to the case $N>1$. Other results on sign changes of Hecke eigenvalues include
\cite{WK07, KS07, PS08, RSW16, DK18, GKP21}. The first work that
showed that the sequence $\eta(p)$ for
$p$ \emph{prime }has infinitely many sign changes is due
to Das \cite{SD13} (however, he only
considered the case $N=1$ and did not consider primes in
arithmetic progressions or bounds for the first sign-change).

For more general automorphic $L$-functions,
Cho and Kim \cite[Theorem 1.1]{CK21}  (building upon work
of Jin \cite{SJ19}) proved a bound for the
first negative coefficients at primes of $L$-function attached to
a self-dual automorphic representation under the assumption
of Ramanujan conjecture. As a direct consequence of \cite[Theorem 1.1]{CK21},
one can find a prime $p\ll k^{7+\epsilon}$ such that $\eta(p) < 0$,
where $\eta(p)$ is as in Theorem \ref{First-sign}.
However, our result in Theorem \ref{First-sign} gives a stronger
bound as we have more precise information about the archimedean
$L$-factor in our case.

\subsection{Main Results on Fourier coefficients}
For
Siegel cusp forms of degree $2$, the question of sign changes for the full sequence of Fourier coefficients
$a(T) ~ (T \in \n_2)$ has been investigated in \cite{SJ08, CGK15, GS17}.
However there are not many previous results for the more subtle question
of sign changes when $T$ is restricted to some
sparse subsequence of $\n_2$. A matrix $T \in \n_2$ is called {\em fundamental} if $ - \text{det }(2T)$
is a fundamental discriminant, and in this case, the Fourier coefficient
$a(T)$ will be called a {\em fundamental Fourier coefficient}. For a fixed
 $T \in \n_2$ and varying $m \in \N$, we call the sequence $a(mT)$ the
{\em radial} Fourier coefficients associated to $T$.  Very recently,
J\"a\"asaari, Lester and Saha \cite[Theorem A]{JLS21+} proved the infinitude of sign changes for
the subsequence of fundamental Fourier
coefficients $a(T)$. Shankhadhar and
Tiwari \cite{ST21+} restricted themselves to certain other types of
Fourier coefficients $a(T)$. In this paper, we investigate signs of
the radial Fourier coefficients $a(pT)$ for fixed fundamental $T$ and $p$
varying over primes. The properties of these coefficients $a(pT)$ are closely connected to those of the Hecke eigenvalues $\eta(p)$.

Our result below (Theorem \ref{t:fcsignchange}) gives a new
characterisation of the forms of classical
Saito--Kurokawa type using sign changes of Fourier coefficients. We note that our result does not require $F$ to be an eigenform. We refer the reader to \cite[Section 8]{pssmb} for some alternative characterisations of Saito-Kurokawa lifts.

\begin{thm}[see Propositions \ref{F-o-Maass}, \ref{Maass}]\label{t:fcsignchange}
Let $k$ and $N$ be positive integers. If $k=2$, assume that the Ramanujan conjecture holds for each non-zero element of $S_2(\Gamma^{(2)}_0(N))$ that is a Hecke eigenform at all good primes and not attached to a CAP representation.

Let $F \in S_k(\Gamma^{(2)}_0(N))$ be a non-zero cusp form having real
Fourier coefficients $a(T)$ and let $T_0$ be a fundamental matrix such that
$a(T_0) \ne 0$ and $\gcd(4\det(T_0), N)=1$.
If $F$ lies in $S_k^*(\G_0^{(2)}(N))$, then for all sufficiently
large primes $p$ the Fourier coefficients $a(pT_0)$ have the same sign. On the other hand, if $F$ is orthogonal to $S_k^*(\G_0^{(2)}(N))$,
then there exist infinitely many primes $p$ such that $a(pT_0) > 0$
and infinitely many primes $q$ such that $a(qT_0) < 0$.
\end{thm}

We remark here that given $F \in S_k(\Gamma^{(2)}_0(N))$, the existence of (infinitely many) fundamental $T_0$ with
$a(T_0) \ne 0$ and $\gcd(4\det(T_0), N)=1$ is known for $N$
squarefree under mild conditions on $F$ at primes dividing $N$;
see Remark \ref{remark:nonvanishing} below.

\medskip

Our final theorem on Fourier coefficients is quite different in flavour from the results described so far and concerns upper bounds on the sizes of the Fourier coefficients.
Given $F \in S_k(\Sp_4(\Z)$, the best unconditional currently known upper bound for $a(T)$  is due to Kohnen \cite{WK92} and states that
$| a(T) | \ll_{F, \epsilon} ({\rm det}~T)^{\frac{k}{2} - \frac{13}{36} +\epsilon}$. However, this result is rather far from the expected truth.
The famous Resnikoff--Salda\~na conjecture \cite[Conjecture IV]{RS74} predicts that for $F \in S_k(\Sp_4(\Z))$
orthogonal to forms of classical Saito-Kurokawa type, the Fourier coefficients satisfy the bound
$| a(T) | \ll_{F, \epsilon} ({\rm det}~T)^{\frac{k}{2} - \frac{3}{4} +\epsilon}$
for any $\epsilon > 0$. This conjecture is extremely deep and difficult, and appears to be out of reach at present.
The article \cite{DK13} proved bounds towards this for radial Fourier
coefficients while in \cite{JLS21+}, the authors considered $F \in S_k(\Gamma^{(2)}_0(N))$ with $N$ squarefree,
and assuming the Generalized Riemann Hypothesis proved the bound  $|a(T)| \ll_{F, \eps} \frac{\det(T)^{\frac{k}2 - \frac{1}{2}}}{ (\log |\det(T)|)^{\frac18 - \eps}}$ for fundamental matrices $T$.

For $F$ lying in the generalized Maass subspace $S_k^*(\G_0^{(2)}(N))$, the situation is slightly different. In this case  one expects that the Resnikoff--Salda\~na bound
$| a(T) | \ll_{F, \epsilon} ({\rm det}~T)^{\frac{k}{2} - \frac{3}{4} +\epsilon}$ still holds for the fundamental Fourier
coefficients $a(T)$ of $F$ (this can be shown to follow from the Generalized Lindel\"of hypothesis; see Proposition \ref{p:fundSK}); however, it is known that this bound \emph{fails} to hold for general (non-fundamental) coefficients. Indeed, Kohnen \cite{Kohnen04} (see also B\"ocherer--Raghavan \cite{BR88}) observed that for a Saito-Kurokawa lift of full level, infinitely many of the Fourier coefficients $a(LT_0)$ $(L \ge 1)$ do \emph{not} satisfy the Resnikoff--Salda\~na bound. This was refined by Das and Kohnen \cite{DK13} who showed that given a Saito-Kurokawa lift of full level, there exists a fundamental $T \in \n_2$
such that $|a(pT)| \gg_{F,T} p^{k-1}$ for $p$ lying in a set of primes  of positive density. Therefore, the Resnikoff--Salda\~na conjecture needs to be modified for forms in the Maass subspace. For the case $N=1$, Das and Kohnen \cite[section 4]{DK13} conjectured  that the bound
\begin{equation}\label{DK}| a(T) |
~\ll_{F, \epsilon}~
\big( \det T \big)^{\frac{k-1}{2}+\epsilon}\end{equation}
 holds for all Fourier coefficients $a(T)$ of a form lying in the Maass subspace.  From the examples above, it is clear that the exponent $\frac{k-1}{2}$ in \eqref{DK} is best possible.  Our next result shows that \eqref{DK} holds for a form lying in the generalized Maass subspace of level $N$ under a mild assumption on $\det(T)$ which is always satisfied when $N$ is squarefree.

\begin{thm}\label{DK-SK}
Let $k$ and $N$ be positive integers. Let $F \in S_k^*(\Gamma^{(2)}_0(N))$ be a non-zero cusp form having
Fourier coefficients $a(T)$. Then, for any $\epsilon > 0 $, the bound
$$
| a(T) |
~\ll_{F, \epsilon}~
\big( \det T \big)^{\frac{k-1}{2}+\epsilon}
$$ holds for all $T \in \n_2$ which have the property that $\gcd(4\det(T), N)$ is squarefree.
\end{thm}

In fact, our result is sharper than stated in Theorem \ref{DK-SK}, and gives a more refined bound on the size of the Fourier coefficient; see Theorem \ref{t:SKRS}. In particular, Theorem \ref{t:SKRS} implies that for all primitive $T$ (i.e., there being no common factor dividing the entries of $T$) we have  $| a(T) | \ll_{F, \epsilon} ({\rm det}~T)^{\frac{k}{2} - \frac{3}{4} +\epsilon}$ under the Generalized Lindel\"of hypothesis. This shows that the counterexamples to the Resnikoff--Salda\~na bound essentially arise from the matrices of the form $LT_0$ as in the examples of Kohnen and Das--Kohnen noted above.

After a previous version of this manuscript was made available, we became aware (thanks to Patrick Dynes) of a recent preprint by Ikeda and Katsurada \cite{Ikkat22} where they prove a similar bound for the Fourier coefficients in the overlapping context of an Ikeda lift of full level. Their methods are quite different from ours and use the Gross--Keating invariant.

\subsection{Methods}
Theorems \ref{First-sign} and \ref{split-inert} rely on properties of the underlying $L$-functions. More precisely, for Theorem \ref{First-sign}, we build upon a method
of Jin \cite[Theorem 3.1]{SJ19}
(also see Cho and Kim \cite{CK21})  and for Theorem \ref{split-inert} we extend
techniques of Kohnen--Lau--Wu \cite{KLW13} by adding a character twist. For both these results, it is useful to separate into cases according to the various
global Arthur packets that $F$ can correspond to
(see Proposition \ref{p:lfunctionsfacts}). If $F$ is attached to a CAP
representation, then we are essentially reduced to properties of Dirichlet characters, which gives us stronger
versions of the desired results. On the other hand, if $F$ is not attached to a CAP representation,
then we use the transfer to $\GL_4$ and properties of the (twists of the) $\GL_4$ $L$-functions to obtain these theorems.

The proof of Theorem \ref{t:fcsignchange} relies on two main
ingredients. The first is an identity relating Fourier coefficients
and Hecke eigenvalues (see \propref{F-H relation}). The second ingredient
is a Selberg orthogonality style relation for \emph{linear combinations}
of Hecke eigenvalues (Proposition \ref{prop:selberggen}) which exhibits a certain dichotomy due to the nature of Arthur packets on $\GSp_4$ and may be of independent interest.

For \thmref{DK-SK} we use quite different methods from the above. In fact, we give two different proofs of this theorem. The first  proof uses representation-theoretic
methods. Precisely, we use the spherical $p$-adic Bessel functions to relate the size of $a(T)$ to that of $a(T_0)$ where $T_0$ is a fundamental matrix. This reduction requires delicate estimates for the sizes of relevant values of these $p$-adic Bessel functions. The desired bound for $a(T_0)$  is a consequence of  the convexity bound for certain $L$-functions. As noted earlier, this method gives us a sharper result than we have stated above; see Theorem \ref{t:SKRS} and Proposition \ref{t:SKRSnew} for the more refined bounds.

The second proof is only valid for
squarefree $N$ and is easy to derive from recent developments in the theory
of Fourier--Jacobi coefficients of a Siegel cusp form. This proof exploits bounds relating  the Petersson norm of a Jacobi form, the sizes of its Fourier coefficients, and the index of the Jacobi form. We decided to
include both proofs here as they use very different ideas and the second proof highlights
the importance (as well as the limitations) of the theory of Jacobi forms
in studying Siegel modular forms. See Section~5 for further details.

Some key preliminary facts will be listed in Section \ref{s:prelim}.
In sections 3 -- 6, we give proofs of
Theorems~\ref{First-sign} -- \ref{DK-SK}.

\subsection{Notations}
Throughout the paper, $\epsilon > 0$
will denote a sufficiently small real number whose precise value may change
in each appearance. The notation
$A \ll_{x,y,z} B$ or $A = O_{x,y,z}(B)$ will mean there exists a positive
constant $C$ depending at most on $x, y, z$ such that $|A| \le C |B|$.

We say that an integer $d$ is a fundamental discriminant if $\Q(\sqrt{d})$
is a field of discriminant $d$. For a fundamental discriminant $d$, we let
$\chi_d$ be the associated quadratic Dirichlet character. For positive integers $a$, $b$, we use the notation $\gcd(a, b^\infty)$ to denote the limit $\lim_{N \rightarrow \infty} \gcd(a, b^N)$, which is clearly equal to $\gcd(a, b^N)$ for all sufficiently large $N$. For a prime $p$, we let $v_p(a)$ denote the largest non-negative integer $v$ such that $p^v$ divides $a$.

We use $\A$ to denote the ring of adeles over $\Q$. The notation
$L(s, \pi )$  for the $L$-function of an automorphic representation $\pi$  will mean the finite part of the $L$-function (i.e., without
the archimedean factors) and the notation $\Lambda(s, \pi)$ will
denote the completed $L$-function including the archimedean
factors. All $L$-functions will be normalized so that the functional
equation takes $s \mapsto 1-s$. Given an $L$-function with an
Euler factor decomposition $L(s, \pi) = \prod_p L(s, \pi_p)$, we
use the notation $L^N(s, \pi):=\prod_{p \nmid N}L(s, \pi_p)$ for
each integer $N$ to denote the partial $L$-function away from
primes dividing $N$.

By $1_2$ we mean the identity matrix of size $2$.
We denote by $J$ the $4$ by $4$ matrix given by
$
J =
\left(\begin{smallmatrix}
0 & 1_2\\
-1_2 & 0\\
\end{smallmatrix}\right),$ and we define the algebraic groups
$\GSp_4$ and $\Sp_4$ by
$$
\GSp_4(R) = \{g \in \GL_4(R) \; | \; g^tJg =
  \mu_2(g)J,\:\mu_2(g)\in R^{\times}\},$$
$$
\Sp_4(R) = \{g \in \GSp_4(R) \; | \; \mu_2(g)=1\},
$$
for any commutative ring $R$.  The Siegel upper-half space $\H_2$ of degree 2 is defined by
$$
\H_2 = \{ Z \in \mathrm{Mat}_{2\times 2}(\C)\;|\;Z =Z^t,\ \Im(Z)
  \text{ is positive definite}\}.
$$

The subgroup $\GSp_4(\R)^+$ of $\GSp_4(\R)$ consists of the matrices $g$ such that $\mu_2(g)>0$. For $g =\left(\begin{smallmatrix} A&B\\ C&D \end{smallmatrix}\right) \in \GSp_4(\R)^+$ and $Z \in \H_2$, we define $g \langle Z\rangle = (AZ+B)(CZ+D)^{-1}$ and we denote $J(g,Z) = CZ + D$. For a function $F: \H_2 \rightarrow \C$, a matrix $g \in \GSp_4(\R)^+$, and an integer $k$, we define the function $F|_kg: \H_2 \rightarrow \C$ by $(F|_kg)(Z) = \mu(g)^k \det(J(g, Z))^{-k} F(g \langle Z\rangle)$.

\subsection{Acknowledgments}
We thank Ariel Weiss and Patrick Dynes for comments. At the first stages of this work, BP was a JSPS postdoctoral fellow at Kyushu University, Japan and was
supported by JSPS KAKENHI Grant No. 19F19318.
He would like to thank his academic host Professor Masanobu Kaneko for the support.
This work is supported by the Engineering and Physical Sciences Research Council [grant number EP/W001160/1].
%For the purpose of open access, the authors have applied a Creative Commons Attribution (CC BY) licence to any Author Accepted Manuscript version arising from this submission.

\section{Preliminaries}\label{s:prelim}

\subsection{Basics of Siegel cusp forms}
Let $k$ and $N$ be positive integers. Let $S_k(\Gamma^{(2)}_0(N))$ denote the space of Siegel cusp forms of weight $k$ and
of degree $2$ for the group $\G^{(2)}_0(N)$
defined by
\begin{equation}\label{defu1n}
 \Gamma^{(2)}_0(N) = \Sp_4(\Z) \cap \left(\begin{smallmatrix}\Z& \Z&\Z&\Z\\\Z& \Z&\Z&\Z\\N\Z& N\Z&\Z&\Z\\N\Z&N \Z&\Z&\Z\\\end{smallmatrix}\right).
\end{equation}
Recall that the elements of $S_k(\Gamma^{(2)}_0(N))$ consist of holomorphic functions $F$ on
$\H_2$ which satisfy the relation
\begin{equation}\label{siegeldefiningrel}
F|_k \gamma= F \quad  \text{for all }\gamma \in \Gamma^{(2)}_0(N),
\end{equation}
and vanish at all the
cusps.  For a precise formulation of this cusp vanishing condition, see~\cite[I.4.6]{Fr1991};
 for definitions and basic properties
of Siegel modular forms we refer the reader to \cite{andzhu}. Given $F_1$ and $F_2$ in $S_k(\Gamma^{(2)}_0(N))$, we define their Petersson inner product by
\begin{equation}\label{eqn:petersson-def}
\langle F_1, F_2\rangle
=
\frac{1}{[\Sp(4,\Z):\Gamma_0^{(2)}(N)]}\;
\int\limits_{\Gamma_0^{(2)}(N) \bs \H_2} F_1(Z) \overline{F_2(Z)} (\det Y)^{k - 3}\,dX\,dY.
\end{equation}

For any positive integer $m \in \N$ with $\gcd(m,N)=1$,
the Hecke operator $T(m)$ acting on the space $S_k(\Gamma^{(2)}_0(N))$ is defined by
$$
T(m) F
~=~
m^{k-3}\sum_{\gamma \in \G_0^{(2)}(N) \backslash \mathcal O_{2, m}(N)}
F |_k \gamma,
$$
where
$$
\mathcal O_{2, m}(N)
~:=~
\left\{
\gamma\in  \left(\begin{smallmatrix}\Z& \Z&\Z&\Z\\\Z& \Z&\Z&\Z\\N\Z& N\Z&\Z&\Z\\N\Z&N \Z&\Z&\Z\\\end{smallmatrix}\right)  ~|~ \gamma^t J \gamma = m J
\right\}.
\phantom{m}
$$
We say that $F \in S_k(\Gamma^{(2)}_0(N))$ is a \emph{Hecke eigenform at all good primes} if there exist complex numbers
$\eta(m)$ such that $T(m)F = \eta(m) F$ for all $m \in \N$ satisfying $\gcd(m,N)=1$. Using the self-adjointness of the Hecke operators (see, e.g., Proposition 1.8 of Chapter 4 of \cite{andzhu}) it follows that the Hecke eigenvalues $\eta(m)$ associated to a Hecke eigenform are in fact all \emph{real} numbers.
As in the case of elliptic modular forms, the Hecke eigenvalues
$\eta(m)$ are multiplicative. We also use the normalized Hecke eigenvalues which are defined as:
\begin{equation}\label{n-Hecke}
\lambda(m) ~:=~ \frac{\eta(m)}{m^{k - 3/2}}
\phantom{m}\text{for}\phantom{m}
m \in \N.
\end{equation}

\subsection{Automorphic representations of $\GSp_4(\A)$ and $L$-functions}\label{s:autprelim}
Given $F \in S_k(\Gamma^{(2)}_0(N))$, we let $\pi_F$ denote the representation of $\GSp_4(\A)$ generated by the adelization (in the sense of \cite[\S3]{sahapet}) of $F$. The representation $\pi_F$ is of trivial central character and can be written as a direct sum of  irreducible, cuspidal, automorphic representations of $\GSp_4(\A)$ of trivial central characters. We denote by $\Pi(F)$ the set of irreducible automorphic representations of $\GSp_4(\A)$ that occur in the above direct sum decomposition of $\pi_F$. Clearly all elements of $\Pi(F)$ are also of trivial central character. As a consequence of multiplicity one for $\GSp_4$ due to Arthur \cite{JA13}, it follows that if $\pi \in \Pi(F)$  then $\pi$ occurs in $\pi_F$ with multiplicity one. We say that an element $F$ of $S_k(\Gamma^{(2)}_0(N))$ gives rise to an irreducible representation if $\pi_F$ is irreducible, or equivalently, if $\Pi(F)$ is a singleton set.

It can be seen easily \cite[Prop. 3.12]{sahapet} that if $F$ gives rise to an irreducible representation then $F$ is a Hecke eigenform at all good primes. While the converse is not true (due to failure of strong multiplicity one), we have the following lemma which clarifies the situation.
\begin{lem}\label{l:equivconds}Let $F \in S_k(\Gamma^{(2)}_0(N))$. Then the following conditions are equivalent:
\begin{enumerate}
\item \label{c1} If $\pi_1$ and $\pi_2$ are any two elements of $\Pi(F)$ then $\pi_{1,p} \simeq \pi_{2,p}$ for all primes $p \nmid N$.

\item \label{c2} If $\pi_1$ and $\pi_2$ are any two elements of $\Pi(F)$ then $\pi_{1,p} \simeq \pi_{2,p}$ for almost all primes $p \nmid N$.

\item \label{c3}$F$ is a Hecke eigenform at all good primes, i.e., all primes not dividing $N$.

\item \label{c4} $F$ is a Hecke eigenform at almost all good primes, i.e., there exits a multiple $N'$ of $N$ such that $F$ is an eigenform for the Hecke operators $T(m)$ for all $m \in \N$ satisfying $\gcd(m,N)=1$.
\end{enumerate}
\end{lem}
\begin{proof}The equivalence of \eqref{c1} and \eqref{c2} follows from Lemma 3.1.2 of \cite{sch05}, the fact that $\pi_{i,p}$ is spherical for all $p \nmid N$, and the fact that there is a unique spherical constituent in each induced representation. The equivalence of \eqref{c1} and \eqref{c3}, as well as of \eqref{c2} and \eqref{c4},  follows from Proposition 3.12 of \cite{sahapet}.
\end{proof}

Given $F \in S_k(\Gamma^{(2)}_0(N))$, and $\pi \in \Pi(F)$, we can take the projection $\phi'_F$ of the adelization $\phi_F$ of $F$ onto the $\pi$-isotypic subspace of $\pi_F$, and then unadelize $\phi'_F$ to obtain some $F' \in S_k(\Gamma^{(2)}_0(N))$ that is a Hecke eigenform at all good primes  and whose adelization generates the space $V_\pi$ of $\pi$. \emph{Therefore, given an  irreducible cuspidal automorphic representation $\pi$ of $\GSp_4(\A)$, there exists $F \in S_k(\Gamma^{(2)}_0(N))$ such that $\pi = \pi_F$ if and only if there exists $F \in S_k(\Gamma^{(2)}_0(N))$ such that $\pi \in \Pi(F)$.} If $\pi$ has the above property, we say that $\pi$ arises from $S_k(\Gamma^{(2)}_0(N))$.

%We say that an irreducible cuspidal automorphic representation $\pi$ of $\GSp_4(\A)$ arises from $S_k(\Gamma^{(2)}_0(N))$  if there exists $F \in S_k(\Gamma^{(2)}_0(N))$ whose adelization generates $V_\pi$ %\fixme{$V_\pi$ hasn't been defined. Now defined in notations}
    %(in which case, by definition, $F$ gives rise to the irreducible representation $\pi$, which therefore must be of trivial central character by the comments above).

    We say that an irreducible cuspidal automorphic representation $\pi$ of $\GSp_4(\A)$ is CAP if it is nearly equivalent %\fixme{please add definition/reference, (I assume this is as in \cite{sahapet})AS: see notations }
to a constituent of a global induced representation of a proper parabolic subgroup of $\GSp_4(\A)$. If $F \in S_k(\Gamma^{(2)}_0(N))$ is a Hecke eigenform at all good primes, we say that $F$ is CAP (or $F$ is attached to a CAP representation) if some (equivalently, every) $\pi \in \Pi(F)$ is CAP.

Given an irreducible cuspidal automorphic representation $\pi$ of $\GSp_4(\A)$ of
trivial central character, we let
\begin{equation}\label{Spinor-D}
L(s, \pi) = \prod_{p<\infty}L(s, \pi_p)
\end{equation}
denote the (finite part of the) degree $4$ (spinor) $L$-function associated to $\pi$, where
the local factors $L(s, \pi_p)$ are defined using the local Langlands
correspondence \cite{GT11} for $\GSp_4$. For each prime $p \nmid N$,
there are complex numbers (known as Satake parameters) $\alpha_{p,i}$, $1\le i \le 4$ satisfying
$\alpha_{p,1} \alpha_{p,2} = \alpha_{p,3} \alpha_{p,4} = 1$ such
that $L(s, \pi_p) = \prod_{i = 1}^4 \left( 1 - \frac{\alpha_{p, i}}{p^s} \right)^{-1}$. By the work of Arthur \cite{JA13} and using the accidental isomorphism $\mathrm{PGSp}_4 \simeq \mathrm{SO}_5$, we know that $L^N(s, \pi) = L^N(s, \Pi)$ where $\Pi$ is an isobaric automorphic representation of $\GL_4(\A)$.

Given two irreducible cuspidal automorphic representation $\pi_1$
and $\pi_2$ of $\GSp_4(\A)$ of trivial central characters, we let
\begin{equation}\label{Rankin}
L(s, \pi_1\times \pi_2) = \prod_{p<\infty} L(s, \pi_{1,p}\times \pi_{2,p})
\end{equation}
denote the associated degree $16$ Rankin-Selberg product of the spin $L$-functions of $\pi_1$ and $\pi_2$, where each local factor is defined using the local Langlands correspondence. For each prime $p \nmid N$, we have $L(s, \pi_{1,p} \times \pi_{2,p}) = \prod_{1\le i,j \le 4}\left( 1 - \alpha_{p,i}\beta_{p,j}p^{-s} \right)^{-1}$.  As before, using work of Arthur, we have that $L^N(s, \pi_1 \times \pi_2) = L^N(s, \Pi_1 \times \Pi_2)$ where $\Pi_1$ and $\Pi_2$ are isobaric automorphic representations of $\GL_4(\A)$. By the general theory of Rankin-Selberg $L$-functions on $\GL_n$, the $L$-function $L(s, \pi_1 \times \pi_2)$ has analytic continuation to $\C$
as a meromorphic function which is non-vanishing
on $\Re(s) = 1$.

 Let $F \in S_k(\Gamma^{(2)}_0(N))$ be a Hecke eigenform at all good primes and let $\pi \in \Pi(F)$. Then we have the following relation between the degree $4$ $L$-function of $\pi$ and the normalized Hecke eigenvalues of $F$: \begin{equation}\label{Spinor}
L^N(s, \pi)
~=~
\zeta^N(2s+1) \sum_{\substack{m = 1 \\ (m,N)=1}}^\infty \frac{\lambda(m)}{m^s}
~=~
\prod_{p\nmid N} \prod_{i = 1}^4 \left( 1 - \frac{\alpha_{p, i}}{p^s} \right)^{-1}.
\end{equation}

\subsection{Arthur packets}\label{s:arthur} Cuspidal automorphic representation of $\GSp_4$ with trivial central character can be naturally viewed as representations of  $\PGSp_4 \simeq \SO_5$.  Arthur \cite{JA13} has given a
classification of the discrete automorphic spectrum of $\SO_5$ into \emph{Arthur packets} in terms of automorphic representations of
general linear groups. In the following proposition, we state this classification for the $\pi$ which arise from $S_k(\Gamma^{(2)}_0(N))$ and write some key properties  for each packet.

\begin{prop}\label{p:lfunctionsfacts}Suppose that $F \in S_k(\Gamma^{(2)}_0(N))$ is a Hecke eigenform at all good primes. Let the normalized Hecke eigenvalues $\lambda(m)$ of $F$ be as defined in \eqref{n-Hecke}. Then exactly one of the following cases must occur.
\begin{enumerate}
\item\label{gen}$F$ is of \emph{general} type: This case can only occur if $k\ge 2$. There exists an irreducible \emph{cuspidal} automorphic representation $\Pi$ of $\GL_4(\A)$ with trivial central character such that each $\pi \in \Pi(F)$ has a strong functorial lifting to $\Pi$. In this case, $L(s, \pi)=L(s, \Pi)$ is entire and $L(s, \pi \times \pi) = L(s, \Pi\times \Pi)$ has a simple pole at $s=1$. The representation $\Pi$ is unramified at the finite primes not dividing $N$.

\item $F$ is of \emph{Yoshida} type: This case can only occur if $N>1$ and $k\ge 2$. There exists an  automorphic representation $\pi_1$ of $\GL_2(\A)$ attached to a cuspidal newform $f_1$ of weight $2$ for $\Gamma_0(N_1)$ and an  automorphic representation $\pi_2$ of $\GL_2(\A)$ attached to a cuspidal newform $f_2$ of weight $2k-2$ for $\Gamma_0(N_2)$, with $\pi_1$ and $\pi_2$ non-isomorphic, such that for each $\pi \in \Pi(F)$, we have $L(s, \pi) = L(s, \pi_1)L(s, \pi_2)$. Thus, the functorial lift of the representation $\pi$ to $\GL_4(\A)$ is the isobaric sum of two distinct, unitary, cuspidal automorphic representations of $\GL_2(\A)$ with trivial central characters and $L(s, \pi)$ is entire. Any prime $p$ dividing $N_1N_2$ must divide $N$. For each prime $p$ coprime to $N$ we have $\lambda(p) = \lambda_1(p) + \lambda_2(p)$ where $\lambda_i(p)$ is the normalized Hecke eigenvalue of $f_i$ at the prime $p$. We have the factorization \[L(s, \pi \times \pi) = L(s, \pi_1 \times \pi_2)^2 \zeta(s)^2 L(s, \sym^2 \pi_1) L(s, \sym^2 \pi_2)\] and hence $L(s, \pi \times \pi)$ has a pole of order $2$ at $s=1$.

 \item \label{case:SK} $F$ is of \emph{Saito--Kurokawa} type: This case can only occur if $k \ge 2$; moreover if $N$ is squarefree, then $k$ must be even. There exists a representation $\pi_0$ of $\GL_2(\A)$ of trivial central character  attached to a cuspidal holomorphic newform of weight $2k-2$ and a primitive Dirichlet character $\chi_0$ satisfying $\chi_0^2=1$ such that \[L^N(s, \pi) = L^N(s, \pi_0) L^N(s+1/2, \chi_0)  L^N(s-1/2, \chi_0)\] for each representation $\pi \in \Pi(F)$. In particular, each such $\pi$ is CAP with respect to the Siegel parabolic.  The representation $\pi_0$ and the character $\chi_0$ are unramified at the finite primes not dividing $N$. For each prime $p$ coprime to $N$ we have $\lambda(p) = \lambda_0(p) + p^{1/2}\chi_0(p) + p^{-1/2}\chi_0(p)$ where $\lambda_0(p)$ is the normalized Hecke eigenvalue of $\pi_0$ at the prime $p$.

     If $F$ is of Saito--Kurokawa type with $\chi_0$ trivial (which is \emph{always} the case if $N$ is squarefree), we will say that $F$ is of \emph{\textbf{classical Saito-Kurokawa}} type.

 \item $F$ is of \emph{Soudry} type: This case can only occur if $k \in \{1,2\}$ and $N>1$. There exists a unitary cuspidal automorphic representation $\pi_0$ of $\GL_2(\A)$  attached to a CM cuspidal holomorphic newform $f_0$ of weight 1, level $N_0$ and character $\xi_0$ such that \[L^N(s, \pi) = L^N(s+1/2, \pi_0) L^N(s-1/2, \pi_0)\] for each $\pi \in \Pi(F)$.  Therefore each representation $\pi \in \Pi(F)$ is CAP with respect to the Klingen parabolic. The character $\xi_0$ is a non-trivial quadratic Dirichlet character of conductor dividing $N_0$, and any prime $p$ dividing $N_0$ must divide $N$.  For each prime $p$ coprime to $N$ we have $\lambda(p) = p^{1/2}\lambda_0(p) + p^{-1/2}\lambda_0(p)$ where $\lambda_0(p)$ is the normalized Hecke eigenvalue of $\pi_0$ at the prime $p$.

 \item $F$ is of \emph{Howe--Piatetski-Shapiro} type: This case can only occur if $k \in \{1,2\}$ and $N>1$. Each representation $\pi \in \Pi(F)$ is CAP with respect to the Borel parabolic. There exist distinct quadratic Dirichlet characters $\chi_1$ and $\chi_2$ such that \[L^N(s, \pi) = L^N(s+1/2, \chi_1) L^N(s-1/2, \chi_1)L^N(s+1/2, \chi_2) L^N(s-1/2, \chi_2).\]The representation $\chi_1$ and $\chi_2$ are unramified at the finite primes not dividing $N$. For each prime $p$ coprime to $N$ we have $\lambda(p) = p^{1/2}\chi_1(p) + p^{-1/2}\chi_1(p) + p^{1/2}\chi_2(p) + p^{-1/2}\chi_2(p).$

\end{enumerate}

\end{prop}
\begin{proof}This follows from the work of Arthur \cite{JA13}; see also the papers of Schmidt \cite{RS18, schcap}. For the convenience of the reader, we explain some of the key points below.

Let $\pi \in \Pi(F)$. Then, as explained in Section 1.1 of \cite{RS18}, the global (Arthur) parameter $\psi$ of $\pi$ in the sense of \cite{JA13} is equal to a formal expression of the form
$\sum_{i=1}^r \mu_i \boxplus \nu(i)$ where $\mu_i$ is a self-dual, unitary, cuspidal, automorphic representation of $\GL_{m_i}(\A)$, $\nu(i)$ is the irreducible representation of $\SL_2(\C)$ of dimension $n_i$, and $\sum_{i=1}^r m_in_i =4$. Using strong multiplicity 1 for $\GL_n(\A)$, it is easy to see that the global parameter does not depend on which $\pi \in \Pi(F)$ we choose; in particular the global  parameter $\psi$ depends only on $F$.  The local component $\pi_v$ at each place $v$ belongs to the local packet associated to the corresponding local parameter $\psi_v$. In particular, this means that if $p \nmid N$, then each $\mu_{i,p}$ is unramified since $\pi_p$ is spherical.

The possibilities for the global parameter can be divided into the following types where we use the notation in \cite{RS18}. We enumerate them in the same order as in the statement of this proposition.

\begin{enumerate}
\item $\psi = \mu \boxtimes 1$, where $\mu$ is a self-dual, symplectic, unitary, cuspidal, automorphic representation of $\GL_4(\A)$.    The assertions follow by taking $\Pi=\mu$. The fact that $\pi$ has a strong lifting to $\Pi$ is just the equivalence of local Arthur parameters and local (Langlands) $L$-parameters in this setting; see the remarks before Theorem 1.1 of \cite{RS18}.    The other assertions follow from standard properties of $L$-functions on $\GL_n(\A)$.

\item  $\psi = \mu_1 \boxtimes 1 \boxplus \mu_2 \boxtimes 1$ where $\mu_1$ and $\mu_2$ are distinct, unitary, cuspidal, automorphic representation of $\GL_2(\A)$ of trivial central characters. We put $\pi_i = \mu_i$. From the archimedean local parameters, we obtain the existence of newforms of weights 2 and $2k-2$ attached to $\pi_1$ and $\pi_2$ as desired. Again, the equivalence of local Arthur parameters and local (Langlands) $L$-parameters in this setting gives the desired expressions for $L(s, \pi)$ and $L(s, \pi \times \pi)$. The formula for the Hecke eigenvalue is immediate from this. Clearly $2k-2 \ge 2$ and so $k\ge 2$.

\item $\psi = (\mu \boxtimes 1) \boxplus (\sigma \boxtimes \nu(2))$ where $\mu$ is a unitary, cuspidal, automorphic representation of $\GL_2(\A)$ of trivial central character and $\sigma$ is a quadratic Dirichlet character. We take $\pi_0=\mu$ and $\chi_0 = \sigma$. The fact that $\pi_0$ corresponds to a holomorphic form of weight $2k-2$, $k \ge 2$ follows immediately by looking at the archimedean local parameters (see the last row of Table 2 of \cite{schcap}). The expression for the $L$-functions (and the Hecke eigenvalues) is a consequence of the corresponding identity relating the local $L$-parameters at all primes not dividing $N$ (see the first row of Table 2 of \cite{schcap}).

    The criterion for $\pi$ to occur in the discrete spectrum is expressed by the sign condition $\prod_v\epsilon(\pi_v) = \epsilon(1/2, \pi_0 \times \chi_0)$ coming from Arthur's multiplicity formula (see the condition above Lemma 1.2 of \cite{schcap}). Above, $\epsilon_v(\pi_v) \in \{\pm 1\}$ is a certain sign occurring in Arthur's work.
     Now assume that $N$ is squarefree. In this case $\chi_0$ must be unramified everywhere by a well-known result of Borel \cite{Borel1976} and hence must be trivial. This implies that in this case $\prod_v\epsilon(\pi_v) = \epsilon(1/2, \pi_0).$ We claim that $k$ must be even in this case. To show this, let $N_0$ be the conductor of $\pi_0$. Then by part (2) of Lemma \ref{l:SK properties} below, $N_0$ divides $N$ and for each $p|N_0$ one of the following possibilities must hold: \begin{itemize}
     \item $\pi_p\simeq \tau(T, \nu^{-1/2})$ is  of type VIb and $\pi_{0,p} = \mathrm{St}_{\GL(2)}$,

      \item  $\pi_p\simeq L((\nu^{1/2}\xi_p\mathrm{St}_{\GL(2)}, \nu^{-1/2})$ is of type Vb and $\pi_{0,p} = \xi_p\mathrm{St}_{\GL(2)}$.
       \end{itemize}

  Using the formulas for $\epsilon(\pi_v)$ in Table 2 of \cite{schcap} we have for all $p|N_0$ that $\epsilon(\pi_p) = \epsilon(1/2, \pi_{0,p})$ (more precisely, they are each equal to -1 in the first possiblity and +1 in the second possibility). If $p$ is a finite prime not dividing $N_0$, we have $\epsilon(\pi_p) = \epsilon(1/2, \pi_{0,p})=1.$ It follows that $\epsilon(\pi_\infty) = \epsilon(1/2, \pi_{0,\infty})$. Since $\pi$ is holomorphic of weight $k$, we have  $\epsilon(\pi_\infty) = -1$ by Table 2 of \cite{schcap}; on the other hand since $\pi_0$ corresponds to a holomorphic form of weight $2k-2$ we have $\epsilon(1/2, \pi_{0,\infty}) = (-1)^{k-1}$. Consequently, $(-1)^{k-1} = -1$ and so $k$ must be even whenever $N$ is squarefree.

\item $\psi = \mu \boxtimes \nu(2)$ where $\mu$ is a unitary, cuspidal, automorphic representation of $\GL_2(\A)$ of non-trivial quadratic central character satisfying $\mu = \mathrm{AI}_{E/\Q}(\theta)$ where $E$ is the quadratic field attached to the central character $\xi_0$ of $\mu$ and $\theta$ is a Hecke character of $E$. From the computation in Table 3 of \cite{schcap}, we see that $k$ equals 1 or 2 and that $\mu$  is attached to a CM cuspidal holomorphic newform of weight 1. Since the conductor $N_0$ of a CM newforms is never equal to 1, and since each prime dividing $N_0$ must divide $N$, it follows that $N>1$.  We take $\pi_0=\mu$ and the expression for the $L$-functions is a consequence of the corresponding identity of local $L$-parameters at all primes not dividing $N$ (see Table 3 of \cite{schcap}).

\item $\psi = (\chi_1 \boxtimes \nu(2)) \boxplus (\chi_1 \boxtimes \nu(2))$ where $\chi_1$ and $\chi_2$ are distinct  quadratic Dirichlet characters. From the computation in Table 1 of \cite{schcap}, we see that $k$ equals 1 or 2 . The expression for the $L$-functions (and as a consequence, the Hecke eigenvalues) follows from the corresponding identity of local $L$-parameters at all primes not dividing $N$ (see Table 1 of \cite{schcap}). Since at least one of the $\chi_i$ is non-trivial, and since each prime dividing the conductor of $\chi_1$ or $\chi_2$ must divide $N$, it follows that $N>1$.
\end{enumerate}
Finally, we note that a result of Weissauer \cite{Weiss92} asserts that if $k=1$ then $F$ must be attached to a CAP representation; in particular, it cannot be of Type 1.
\end{proof}
If $F \in S_k(\Gamma^{(2)}_0(N))$ is a Hecke eigenform at good primes that  is either of Saito--Kurokawa type, or of Soudry type, or of Howe--Piatetski-Shapiro type in the sense defined above then $\pi\in \Pi(F)$ is CAP and so is non-tempered everywhere.

On the other hand if $F \in S_k(\Gamma^{(2)}_0(N))$ is a Hecke eigenform at good primes that is of general type or Yoshida type in the sense defined above,  then one expects that the following conjecture is true.

\begin{conj}[Ramanujan conjecture, currently known for $k \ge 3$ but open for $k=2$] If $F \in S_k(\Gamma^{(2)}_0(N))$ is a Hecke eigenform at good primes that is non-CAP (i.e., of general type or Yoshida type), then for $\gcd(p, N)=1$ we have
$|\alpha_{p, i}| ~=~ 1$ for $1 \le i \le 4$ where $\alpha_{p, i}$ are the Satake parameters defined in \eqref{Spinor}. \end{conj}

A famous result of Weissauer \cite{RW09} implies that $F$ satisfies the Ramanujan conjecture whenever $k \ge 3$. More precisely, for such $F$, each $\pi \in \Pi(F)$ is \emph{tempered} at all good primes, i.e., for $\gcd(p, N)=1$ we have
$|\alpha_{p, i}| ~=~ 1$ for $1 \le i \le 4$ and hence for $\gcd(m, N)=1$, $| \lambda(m)| \le d_5(m)$, where $d_5(m)$ is the number of ways
of writing $m$ as product of $5$ positive integers.
If $k=2$, and $F$ is of Yoshida type, then we also know the Ramanujan conjecture for $F$ from our knowledge of the Ramanujan conjecture for classical holomorphic cusp forms on the upper-half plane.

However, if $k=2$ and $F \in S_2(\Gamma^{(2)}_0(N))$  is a Hecke eigenform at good primes that is of general type, then the Ramanujan conjecture remains open.  For several of our results in this paper, we will need to assume the Ramanujan conjecture in this outstanding case. Note that if $k=1$, then all forms in  $S_1(\Gamma^{(2)}_0(N))$ are CAP.

\begin{rmk}\label{l:weakramweight1}If $F \in S_2(\Gamma^{(2)}_0(N))$ is a Hecke eigenform at good primes that is of general type, then we know the  bound $|\alpha_{p, i}| \le p^{9/22}$ for all primes $p$ with $\gcd(p,N)=1$, which follows from the corresponding bound for $\GL_4$ representations due to Kim--Sarnak \cite{kimsar}.
\end{rmk}

From the self-adjointness of the Hecke operators, it follows that if $F_1$ and $F_2$ correspond to different cases (types) in the sense of Proposition \ref{p:lfunctionsfacts}, then $F_1$ and $F_2$ are orthogonal with respect to the Petersson norm. We let $S_k(\Gamma^{(2)}_0(N))^{\rm T}$ equal the vector space generated by the set of Hecke eigenforms at good primes which are of general type or Yoshida type. We let $S_k(\Gamma^{(2)}_0(N))^{\rm CAP}$ equal the vector space generated by the set of Hecke eigenforms at good primes  $F$ which are non-tempered, i.e., are either of Saito-Kurokawa type or of Soudry type or of Howe-- Piatetski-Shapiro type.
So the space $S_k(\Gamma^{(2)}_0(N))$ has a natural decomposition into orthogonal subspaces $$S_k(\Gamma^{(2)}_0(N))=S_k(\Gamma^{(2)}_0(N))^{\rm  T}  \oplus S_k(\Gamma^{(2)}_0(N))^{ \rm CAP}.$$

We let $S_k^*(\Gamma^{(2)}_0(N)) \subset S_k(\Gamma^{(2)}_0(N))^{ \rm CAP}$ be the subspace generated by the set of Hecke eigenforms  at good primes  which are of classical Saito-Kurokawa type. We will occasionally refer to the space $S_k^*(\Gamma^{(2)}_0(N))$ as the \emph{generalized Maass subspace}. From Proposition \ref{p:lfunctionsfacts} it is clear that if $k>2$ and $N$ is squarefree then $S_k^*(\Gamma^{(2)}_0(N)) = S_k(\Gamma^{(2)}_0(N))^{ \rm CAP}$.

\subsection{Analytic conductors and convexity bounds}
Given a prime $p$ and an irreducible admissible representation $\pi_p$ of $\GSp_4(\Q_p)$, we let $\epsilon(s, \pi_p)$ denote the epsilon factor of $\pi_p$ (with respect to a fixed  additive character $\psi_p$ of level 0), where the epsilon factor is defined using the local Langlands correspondence (in particular, we have $\epsilon(s, \pi_p) = \epsilon(s, \Pi_p)$ where $\Pi_p$ is the representation on $\GL_4(\Q_p)$ that is obtained by local Langlands transfer from $\pi_p$). We let $\epsilon(s, \pi_p \times \pi_p)= \epsilon(s, \Pi_p \times \Pi_p)$ denote the epsilon factor of the Rankin-Selberg $L$-function with respect to the same fixed unramified additive character $\psi_p$. We define the conductor exponents $a(\pi_p)$ and $a(\pi_p \times \pi_p)$ via \[\epsilon(s, \pi_p) = \epsilon(1/2, \pi_p) p^{-a(\pi_p)(s -1/2)},\] \[\epsilon(s, \pi_p \times \pi_p) = \epsilon(1/2, \pi_p \times \pi_p) p^{-a(\pi_p \times \pi_p)(s -1/2)}.\]

Since the conductor exponents above coincide with the corresponding conductor exponents arising from representations of $\GL_4(\Q_p)$, the results of \cite{BH17} and \cite{AC19} imply that there exist absolute positive constants $c_1$, $c_2$ and $d$ such that\footnote{With some bookkeeping, one can show that it suffices to take $c_2=7$, $c_1=2$ and $d=64$.} for any irreducible admissible representation $\pi_p$ of $\GSp_4(\Q_p)$ of trivial central character, we have the bound \begin{equation}\label{e:condbound} c_1 a(\pi_p) -d  \le a(\pi_p \times \pi_p) \le c_2a(\pi_p) \end{equation}

Given an irreducible cuspidal automorphic representation $\pi \simeq \otimes_v \pi_v$ of $\GSp_4(\A)$, we define its finite conductor $q(\pi)$ by  \begin{equation}\label{e:qpidef}q(\pi) = \prod_{p<\infty} p^{a(\pi_p)}.\end{equation} We define the finite conductor of $\pi \times \pi$ by \begin{equation}\label{e:qpipidef}q(\pi \times \pi) = \prod_{p<\infty} p^{a(\pi_p \times \pi_p)}.\end{equation}

 The next lemma shows that the conductor is always bounded by a power of the level. While this result is probably known to experts, we were unable to find a suitable reference.
\begin{lem}\label{l:constants}There exist absolute constants $C$ and $D$ such that given a positive integer $N$ and an automorphic representation $\pi$ of $\GSp_4(\A)$ arising from $S_k(\Gamma^{(2)}_0(N))$, we have $q(\pi) \ll N^C$ and $q(\pi \times \pi) \ll N^D.$
\end{lem}
\begin{proof}Let $\pi$ be as in the Lemma. Then $\pi_p$ is unramified for all $p \nmid N$ and hence $a(\pi_p)=0$ for such $p$. Consider a prime $p$ dividing $N$. Then $\pi_p$ has a vector fixed by \begin{equation}\label{k0pdef}K_{0,p}(N):= \GSp_4(\Z_p) \cap \left(\begin{smallmatrix}\Z_p& \Z_p&\Z_p&\Z_p\\\Z_p& \Z_p&\Z_p&\Z_p\\N\Z_p& N\Z_p&\Z_p&\Z_p\\N\Z_p&N \Z_p&\Z_p&\Z_p\\\end{smallmatrix}\right).
\end{equation} Let $K_{p}(N) \subseteq K_{0,p}(N)$ be the principal congruence subgroup defined by $K_p(N) = \{g \in \GSp_4(\Z_p): g\equiv 1_4 \pmod{N}\}$. It is clear that the lemma follows from \eqref{e:condbound} - \eqref{e:qpipidef} and the following purely local statement:
\emph{There exists an absolute constant $C$ such that for a prime $p$, a positive integer $n_p$ and an irreducible admissible representation $\pi_p$ of $\GSp_4(\Q_p)$ that contains a non-zero $K_p(p^{n_p})$-fixed vector, we have $a(\pi_p)\le C n_p$.}

We now give a proof of the local statement. By the well-known classification theorem of irreducible admissible representations, we have the following two cases \begin{itemize} \item $\pi_p$ is obtained as a subquotient of a parabolically induced representation from one of the parabolic subgroups of $\GSp_4$,
\item $\pi_p$ is a supercuspidal representation.
\end{itemize}
We consider the first case. In this case, the induced representation is induced from a collection of characters of $\GL_1(\Q_p)$ and/or admissible representations of $\GL_2(\Q_p)$, depending on the parabolic subgroup (see Chapter 2 of \cite{NF} for more details). Since the induced representation has a vector fixed by $K_{p}(p^{n_p})$, it follows easily using the Iwasawa decomposition that any character occurring in the inducing data must have conductor exponent at most $n_p$ and any representation of $\GL_2(\Q_p)$ occurring in the inducing data must have a vector fixed by $\{g \in \GL_2(\Z_p): g\equiv 1_2 \pmod{p^{n_p}}\}$ and hence must have conductor exponent at most $2n_p$. Now the local statement follows immediately from the direct expression of $a(\pi_p)$ as a sum of the above conductor-exponents (Table A.9 of \cite{NF}).

We now move on to the second case. In this case, using the fact that $\pi_p$ contains a non-zero $K_p(p^{n_p})$-fixed vector, it is immediate that the formal degree ${\rm deg} (\pi_p)$ \cite[section 13]{GI14} of $\pi_p$  satisfies ${\rm deg} (\pi_p)\le p^{D n_p}$ for some absolute positive constant $D$. On the other hand, the formal degree conjecture (which is proved for $\GSp_4$ by the main result of \cite{GI14} implies that ${\rm deg} (\pi_p) \ge p^{E a(\pi_p \times \pi_p)}$ for some absolute positive constant $E$. This combined with \eqref{e:condbound} concludes the proof of the local statement in this case.
\end{proof}

Suppose $F \in S_k(\Gamma^{(2)}_0(N))^{\rm T}$ is a Hecke eigenform at good primes and  $\pi \in \Pi(F)$. We define
the completed spinor $L$-function (see Table 3 of \cite{Sch17})
\begin{equation}\label{spinLfn}
\Lambda(s, \pi) := (2\pi)^{-2s-k+1} \G(s + k - 3/2) \G(s+1/2) L(s, \pi)
\end{equation}
which coincides with the completed $L$-function of an automorphic representation of $\GL_4(\A)$ by the previous subsection and can therefore be meromorphically extended to the whole complex plane and satisfies the functional equation
$$
\Lambda(s, \pi) ~=~ \epsilon(s, \pi) \Lambda(1-s, \pi).
$$

Using \eqref{spinLfn}, we see that the analytic conductor $\mathfrak q(s, \pi)$ of $L(s, \pi)$
(for definition see \cite[p.95]{IK04}) is
$$
\mathfrak q(s, \pi) \ll q(\pi) k^2 (|s| + 3)^4 \ll_N k^2 (|s| + 3)^4.$$
Therefore the convexity bound for $L(s, \pi)$ in the weight/$t$-aspect is
(see \cite[p.101]{IK04})
\begin{equation}\label{convex-Z}
| L(1/2 + \epsilon + it, \pi ) |
~\ll_{N,\epsilon}~
k^{\frac{1}{2}+\epsilon} (3+|t|)^{1+\epsilon}
\phantom{m}\text{ for any }~~ \epsilon > 0.
\end{equation}

 We will also need the completed $L$-function
$$
\Lambda(s, \pi \times \pi) := L_\infty(s, \pi \times \pi) L(s, \pi \times \pi)
$$
 where
$$
L_\infty(s, \pi \times \pi)  := \Big( (\G_\R(s)\G_\R(s+1)\Big)^2
\Big( \G_\C(s+k-1)\G_\C(s+k-2)\Big)^2
\G_\C(s+1) \G_\C(s+2k-3)
$$
(see \cite[p.375]{DKS14} or \cite[Section 5.2]{PSS14}) with
$\G_\C(s) := 2(2\pi)^{-s} \G(s)$ and $\G_\R(s) := \pi^{-s/2} \G(s/2)$. This also satisfies the usual functional equation taking $s \mapsto 1-s$ by the theory of Rankin-Selberg convolutions on $\GL_n$. By the formula for $L_\infty(s, \pi \times \pi)$, we see that the analytic conductor of $L(s, \pi \times \pi)$ is
$$
\mathfrak q(s, \pi \times \pi) \ll q(\pi \times \pi) k^{10} (|s| + 3)^{16} \ll_N k^{10} (|s| + 3)^{16}.
$$
Hence the convexity bound (see \cite[p.101]{IK04}) for
$L(s, \pi \times \pi)$ in the weight/$t$ aspect is
\begin{equation}\label{convex-R}
| L(1/2 + \epsilon + it, \pi \times \pi) |
~\ll_{N,\epsilon}~
k^{\frac{5}{2}+\epsilon} (3+|t|)^{4+\epsilon}
\phantom{m}\text{ for any }~~ \epsilon > 0.
\end{equation}

\section{Sign changes of Hecke eigenvalues over primes}
In this section we will prove Theorems \ref{First-sign} and \ref{split-inert}.

\subsection{First sign change of Hecke eigenvalues of non-CAP forms}
In this subsection, using a technique of Jin \cite[Theorem 3.1]{SJ19}
(also see Cho and Kim \cite{CK21}) we prove \thmref{First-sign} for
the non-CAP forms.
\begin{prop}\label{non-CAP}
Let $N\ge 1$ be an integer, and let $F \in S_k(\G_0^{(2)}(N))^{ \rm T}$ be a Hecke eigenform
at all good primes with normalized Hecke eigenvalues $\lambda(m)$
for $\gcd(m,N)=1$. Assume that $F$ satisfies the Ramanujan conjecture.\footnote{This is automatic if $k\ge 3$ or $F$ is a Yoshida lift.} Then there exists a prime $p$ with $\gcd(p, N)=1$
and
$$
p \ll_{N, \epsilon} k^{5+\epsilon}
$$
such that $\lambda(p) < 0$.
\end{prop}

\begin{proof}
Fix some $\pi \in \Pi(F)$. For $\Re(s) > 1$, let us define
$$
L^{N,b}(s, \pi)
~:=
\sum_{m \ge 1, \gcd(m, N)=1 \atop m \text{ sq-free}} \frac{\lambda(m)}{m^s}
~=~
\prod_{(p, N) = 1} \left( 1 + \frac{\lambda(p)}{p^s} \right)
$$
and
$$
L^{N,b}(s, \pi\times\pi)
~:=~
\sum_{m \ge 1, \gcd(m, N)=1 \atop m \text{ sq-free}} \frac{\lambda(m)^2}{m^s}
~=~
\prod_{(p, N)=1} \left( 1 + \frac{\lambda(p)^2}{p^s} \right).
$$
Note that the Euler product is valid as $\lambda(m)$ is multiplicative.
For $\Re(s) > 1$, using \eqref{Spinor-D}, \eqref{Spinor} and \eqref{Rankin}, we have
\begin{eqnarray*}
\frac{L^{N,b}(s, \pi)}{L(s, \pi)}
& = &
\prod_{p~|~ N} H_p(p^{-s})
\times
\prod_{(p, N)=1} \left(\left( 1 + \frac{\lambda(p)}{p^s}\right)
\prod_{1 \le i \le 4} \left( 1 - \frac{\alpha_{p,i}}{p^s} \right)\right)
\\
& := &
\prod_{p~|~ N} H_p(p^{-s})
\times
\prod_{(p, N)=1}\left( 1 + \frac{a_2}{p^{2s}} + \cdots +\frac{a_5}{p^{5s}} \right),
\end{eqnarray*}
where $H_p(X)$ is a polynomial of degree $\le 4$ with absolutely bounded coefficients and
\begin{eqnarray}\label{RS-sq-free}
\frac{L^{N,b}(s, \pi\times\pi)}{L(s, \pi\times\pi)}
&=&
\prod_{p~|~N} H_p^{(1)}(p^{-s})
\times
\prod_{(p,N)=1}\left( 1 + \frac{\lambda(p)^2}{p^s}\right)
\prod_{1 \le i, j \le 4} \left( 1 - \frac{\alpha_{p,i}\alpha_{p,j}}{p^s} \right)\\
& :=&
\prod_{p~|~N} H_p^{(1)}(p^{-s})
\times
\prod_{(p, N)=1} \left( 1 + \frac{b_2}{p^{2s}} + \cdots +\frac{b_{17}}{p^{17s}} \right),
\nonumber
\end{eqnarray}
where $a_i, b_j$ for $i = 2, \dots, 5; j = 2, \dots, 17$ are complex
numbers which are bounded by an absolute constant and $H_p^{(1)}(X)$
is a polynomial of degree $\le 16$ with absolutely bounded coefficients.
Hence the Euler products expressing $\frac{L^{N,b}(s, \pi)}{L(s,\pi)}$
and $\frac{L^{N,b}(s, \pi\times\pi)}{L(s, \pi\times\pi)}$ are absolutely and
uniformly convergent in any compact set in the region $\Re(s) > 1/2$.
Therefore, both $\frac{L^{N,b}(s, \pi)}{L(s,\pi)}$ and
$\frac{L^{N,b}(s, \pi\times\pi)}{L(s, \pi\times\pi)}$
define holomorphic functions in $\Re(s) > 1/2$.
%\fbox{Show that these ratios at $s=1$ are $\asymp_N 1$.}

Now, we want to show $|\frac{L^{N,b}(s, \pi\times\pi)}{L(s, \pi\times\pi)}| \asymp_N 1$
at $s=1$. First note that, by Weissauer's bound, none of the Euler
factors of $\frac{L^{N,b}(s, \pi\times\pi)}{L(s, \pi\times\pi)}$
vanishes at $s=1$ any prime $p$. Hence
$\frac{L^{N,b}(s, \pi\times\pi)}{L(s, \pi\times\pi)} \ne 0$ at $s=1$.
Let $C := \text{max}_{2\le i \le 17}|b_i|$. Then $C$ is an absolute constant.
Hence one can easily see from \eqref{RS-sq-free} that
$$
\left| \frac{L^{N,b}(1, \pi\times\pi)}{L(1, \pi\times\pi)}\right| \ll_N 1.
$$
We also have
$$
\left| \frac{L^{N,b}(1, \pi\times\pi)}{L(1, \pi\times\pi)} \right|
\ge
\prod_{p~|~N} |H_p^{(1)}(p^{-1})|
\times
\prod_{(p,N)=1 \atop p\le 16C}\left( 1 - \frac{1}{p} \right)^{16}
\times
\prod_{(p, N)=1 \atop p>16C} \left( 1 - \frac{16C}{p^2}\right).
$$
Therefore
$|\frac{L^{N,b}(1, \pi\times\pi)}{L(1, \pi\times\pi)}| \asymp_N 1$.

Let $w$ be a smooth function on $[0, \infty)$ with support in $[0,2]$ and
$w \equiv 1$ on $[0,1]$ and $0 \le w \le 1$ on $[1, 2]$.
Also let $\widehat{w}(s) := \int_0^\infty w(x)x^{s-1}dx$ be the Mellin transform.
Then, for any real number $X > 1$, using Perron type formula \cite[Section 2.2]{SJ19}
we have
$$
\sum_{m \ge 1, \gcd(m, N)=1 \atop m \text{ sq-free}} \lambda(m)w\left(\frac{m}{X}\right)
~=~
\frac{1}{2\pi i} \int_{1 + \epsilon - i\infty}^{1 + \epsilon + i\infty}
\widehat{w}(s)L^{N,b}(s, \pi) x^s ds
$$
and
$$
\sum_{m \ge 1, \gcd(m, N)=1 \atop m \text{ sq-free}} \lambda(m)^2 w\left(\frac{m}{X}\right)
~=~
\frac{1}{2\pi i} \int_{1 + \epsilon - i\infty}^{1 + \epsilon + i\infty}
\widehat{w}(s) L^{N,b}(s, \pi\times\pi) x^s ds.
$$
Since both $\frac{L^{N,b}(s, \pi)}{L(s,\pi)}$ and $L(s, \pi)$ are holomorphic
in the region $\Re(s) > 1/2$, the function $L^{N,b}(s, \pi)$ has holomorphic
continuation to $\Re(s) > 1/2$. The function $\widehat{w}(s)$ is of rapid decay
as $\Im(s) \to \infty$ and $L(s, \pi)$ is holomorphic in $\Re(s) > 1/2$ implies that
$$
\sum_{m \ge 1, \gcd(m, N)=1 \atop m \text{ sq-free}} \lambda(m)w\left(\frac{m}{X}\right)
~=~
\frac{1}{2\pi i} \int_{1/2 + \epsilon - i\infty}^{1/2 + \epsilon + i\infty}
\widehat{w}(s)L^{N,b}(s, \pi) x^s ds.
$$
The fact $|\frac{L^{N,b}(s, \pi)}{L(s,\pi)}| \ll_{N,\epsilon}~ 1$ in
$\Re(s) ~\ge~ 1/2+\epsilon$ along with the convexity bound \eqref{convex-Z} for
$L(s, \pi)$ gives
\begin{equation}\label{bound-Z}
\sum_{m \ge 1, (m,N)=1 \atop m \text{ sq-free}} \lambda(m)w\left(\frac{m}{X}\right)
~\ll_{N, \epsilon}~
(Xk)^{\frac{1}{2}+\epsilon}.
\end{equation}
Since the poles of $L(s, \pi\times\pi)$ are only at $s = 1$, using convexity bound
\eqref{convex-R} and arguing as above, one can derive
\begin{equation}\label{bound-R}
\sum_{m \ge 1, (m,N)=1 \atop m \text{ sq-free}} \lambda(m)^2 w\left(\frac{m}{X}\right)
~=~
\text{Res}_{s =1} \left( L^{N,b}(s, \pi\times\pi) \right)\widehat{w(1)} X
~+~
O\left( (Xk^5)^{\frac{1}{2}+ \epsilon} \right).
\end{equation}
If $F$ is general type, $L(s, \pi\times\pi)$ has a simple pole
at $s=1$ and
$\text{Res}_{s =1} \left( L^{N,b}(s, \pi\times\pi) \right) ~\gg_N~ k^{-\epsilon}$
for any $\epsilon > 0$ (see \cite[p.928]{CK21}). When $F$ is of Yoshida
type then $L(s, \pi\times\pi)$ has a double pole at $s=1$ with residue
%\fbox{some more details here?}
$$
\text{Res}_{s =1} \left( L^{N,b}(s, \pi\times\pi) \right)
=
2c\gamma \cdot L(1, \pi_1\times\pi_2)^2L(1, {\rm sym}^2\pi_1)L(1, {\rm sym}^2\pi_2),
$$
where $\gamma$ is the Euler constant and $c$ is a constant which
depends on $N$ but not on $k$. Hence, using the results from \cite{Li10},
one can easily see
$\text{Res}_{s =1} \left( L^{N,b}(s, \pi\times\pi) \right) \gg_N~ k^{-\epsilon}$
for any $\epsilon > 0$. Therefore, in both the cases,
$
\text{Res}_{s =1} \left( L^{N,b}(s, \pi\times\pi) \right)\widehat{w(1)} X
\gg_N
(Xk^5)^{\frac{1}{2}+ \epsilon}
$
when $X \gg_N k^{5+\epsilon}$.
Let us assume that $\lambda(m) \ge 0$ for all square-free $m \le 2X$
with $\gcd(m,N)=1$. When $X \gg_N k^{5+\epsilon}$,
Weissauer's bound and estimates \eqref{bound-Z} and \eqref{bound-R} imply
$$
X
~\ll_N~
\sum_{m \ge 1 \atop m \text{ sq-free}} \lambda(m)^2 w\left(\frac{m}{X}\right)
~\ll_N~
X^{\delta} \sum_{m \ge 1 \atop m \text{ sq-free}} \lambda(m)w\left(\frac{m}{X}\right)
~\ll_{N, \epsilon}~
X^\delta (Xk)^{\frac{1}{2}+\epsilon}
$$
for any $\delta > 0$. This is a contradiction to the assumption $X \gg_{N,\epsilon} k^{5+\epsilon}$.
Since $\lambda(m)$ is multiplicative, hence there exists a prime
$p$ with $\gcd(p, N)=1$ and
$$
p ~\ll_{N,\epsilon}~ k^{5+\epsilon}
$$
such that $\lambda(p) < 0$.
\end{proof}

\subsection{Sign changes of Hecke eigenvalues of non-CAP forms over arithmetic progressions}
We begin by recalling the following well-known result, which is a known special case of the Selberg orthogonality conjecture.
\begin{prop}\label{KLW13}Let $\pi$ and $\pi'$ be irreducible unitary cuspidal automorphic representations of $\GL_m(\A)$ and $\GL_{m'}(\A)$ where $m$ and $m'$ are both less than or equal to 4. Assume that at least one of $\pi$ and $\pi'$ is self dual. Let $N$ be any positive integer. Then \[\sum_{p\le X, \ \gcd(p,N)=1} \frac{\lambda_\pi(p)\overline{\lambda_{\pi'}(p)}}{p} = \begin{cases} \log \log X +O_{\pi, \pi', N}(1)&\text{if } \pi' \simeq \pi, \\ O_{\pi, \pi', N}(1) & \text{ otherwise.} \end{cases} \]
\end{prop}
\begin{proof}See Section 2 of \cite{RS96}  or \cite[Corollary 1.5]{LWY05}. Note here that since only finitely many primes divide $N$, the condition $\gcd(p, N)=1$ does not change the asymptotic behaviour of the sums.
\end{proof}

As a corollary of the above we obtain
\begin{cor}\label{cor:selberg}Let $\pi$ be an irreducible unitary cuspidal automorphic self-dual representation of $\GL_m(\A)$ with $m\le 4$ and assume that $\pi$ is not the trivial representation of $\GL_1(\A)$.  Let $a, M$ be positive
integers such that $\gcd(aq(\pi), M)=1$ where $q(\pi)$ is the finite conductor of $\pi$. Let $N$ be any positive integer. Then, for sufficiently large $X$, we have $$
\sum_{p \le X, (p, N)=1 \atop p \equiv a \mod M} \frac{\lambda_\pi(p)}{p}
~=~ O_{\pi, M, N}(1)
\phantom{m}
$$
and
$$
\sum_{p \le X, (p, N)=1 \atop p \equiv a \mod M} \frac{\lambda_\pi(p)^2}{p}
~=~
\frac{1}{\varphi(M)} \log\log X ~+~ O_{\pi, M, N}(1)
$$
\end{cor}
\begin{proof}
For $\gcd(a, M) = 1$, by orthogonality of characters, we have
$$
\sum_{\chi \mod M} \chi(a)^{-1}\chi(p)
~=~
\begin{cases}
\varphi(M) & \text{if }~~ p\equiv a \mod M\\
0 & \text{otherwise}.
\end{cases}
$$
Therefore, we have
$$
\sum_{p \le X, (p, N)=1 \atop p \equiv a \mod M} \frac{\lambda_\pi(p)}{p}
~=~
\frac{1}{\varphi(M)}\sum_{\chi \mod M} \chi(a)^{-1}
\sum_{p \le X \atop (p, NM)=1} \frac{\lambda_{\pi \times \chi}(p)}{p}.
$$
We apply Proposition \ref{KLW13} to the inner sum, which gives us the first part of the proposition.

For the second part, observe that $\pi\not\simeq\pi \times \chi$ unless $\chi$ is trivial. This follows from the fact that the conductors are different: since there exists a prime $p$ dividing the conductor of $\chi$ but not dividing $q(\pi)$, we see that the $p$-part of the conductor of $\pi$ is trivial but the  $p$-part of the conductor of $\pi \times \chi$ is non-trivial. As
$$
\sum_{p \le X, (p, N)=1 \atop p \equiv a \mod M} \frac{\lambda_\pi(p)^2}{p}
~=~
\frac{1}{\varphi(M)}\sum_{\chi \mod M} \chi(a)^{-1}
\sum_{p \le X \atop (p, NM)=1} \frac{\lambda_\pi(p)\lambda_{\pi\times \chi}(p)}{p},
$$
we apply Proposition \ref{KLW13} to the inner sum to complete the proof.
\end{proof}

\begin{lem}\label{inert-prop}
Let $F \in S_k(\G_0^{(2)}(N))^{ \rm T}$ be a Hecke eigenform
at all good primes with normalized Hecke eigenvalues
$\lambda(m)$ for $\gcd(m,N)=1$. Also let $a, M$ be positive
integers such that $\gcd(aN, M) = 1$. Then, for sufficiently large $X$, we have
$$
\sum_{p \le X, (p, N)=1 \atop p \equiv a \mod M} \frac{\lambda(p)}{p}
~=~ O_{F, M}(1)
\phantom{m}
$$
and
$$
\sum_{p \le X, (p, N)=1 \atop p \equiv a \mod M} \frac{\lambda(p)^2}{p}
~=~
\begin{cases}
\frac{1}{\varphi(M)} \log\log X ~+~ O_{F, M}(1) & \text{ if } F \text{ is of general type};\\
\frac{2}{\varphi(M)} \log\log X ~+~ O_{F, M}(1) & \text{ if } F \text{ is of Yoshida type}.
\end{cases}
$$
\end{lem}

\begin{proof}
This follows immediately from Proposition \ref{KLW13}, Corollary \ref{cor:selberg} and the fact that $$
\lambda(p)
~=~
\begin{cases}
\lambda_\Pi(p) & \text{ if } F \text{ is of general type};\\
\lambda_{\pi_1}(p) + \lambda_{\pi_2}(p) & \text{ if } F \text{ is of Yoshida type}.
\end{cases}
$$  where $\Pi$ is an irreducible unitary cuspidal automorphic self-dual representation of $\GL_4(\A)$ and $\pi_1, \pi_2$ are non-isomorphic unitary cuspidal automorphic self-dual representation of $\GL_2(\A)$.
\end{proof}

\begin{prop}\label{G-Y}
Let $F \in S_k(\G_0^{(2)}(N))^{ \rm T}$ be a Hecke eigenform
at all good primes with normalized
Hecke eigenvalues $\lambda(m)$ for $\gcd(m,N)=1$. If $k=2$, then assume further that $F$ satisfies the Ramanujan conjecture.
Also let $a, M$ be positive integers such that $\gcd(aN, M) = 1$.
Then, the sequence $\{ \lambda(p) \}_{p \equiv a \mod M}$
has infinitely many positive and infinitely many negative elements.
\end{prop}

\begin{proof}
We prove the proposition when $F$ is of general type. The
case of Yoshida type follows same argument with obvious
modifications. For $0 < \delta < 1$, let $ y = X^\delta$.
Without loss of generality, assume that $\lambda(p) > 0$
for any prime $ y \le p \le X$ and $p\equiv a \mod M$.
By Weissuaer's estimate $|\lambda(p)| \le 4$.
Therefore \lemref{inert-prop} implies
$$
\frac{1}{\varphi(M)} \log\left( \frac{\log X}{\log y}\right) ~+~ O_{F, M}(1)
~=~
\sum_{y \le p \le X, (p, N)=1 \atop p \equiv a \mod M} \frac{\lambda(p)^2}{p}
~\le~
4 \sum_{y \le p \le X, (p, N)=1 \atop p \equiv a \mod M} \frac{\lambda(p)}{p}
~=~ O_{F, M}(1).
$$
This is impossible for sufficiently small $\delta > 0$ which completes the proof.
\end{proof}

\subsection{The case of CAP forms}
In the case of CAP forms we can prove a stronger result. \begin{prop}\label{CAP}
Let $F \in S_k(\G_0^{(2)}(N))^{ \rm CAP}$ be a Hecke eigenform
at all good primes with normalized Hecke eigenvalues $\lambda(m)$
for $\gcd(m,N)=1$. Assume that $F$ does not lie in $S_k^*(\Gamma^{(2)}_0(N))$.
Then there exists a prime $p$ with $\gcd(p, N)=1$ and
$$
p ~\ll_N~ 1
$$
such that $\lambda(p) < 0$.
\end{prop}
\begin{proof}
We go through the various types of $F$ using the terminology
of Proposition \ref{p:lfunctionsfacts}. First, let $F$ be of
Saito--Kurokawa type with $\chi_0 \ne 1$.
For any prime $\gcd(p, N)=1$, we have
$$
\lambda(p) = \lambda_0(p) + \left( p^{1/2}+p^{-1/2} \right)\chi_0(p),
$$
where $\lambda_0(p), \chi_0$ are as in \propref{p:lfunctionsfacts}.
Since $|\lambda_0(p)| \le 2$, the sign of $\lambda(p)$ is the
same as that of $\chi_0(p)$. Hence the bound for first sign
change does not depend on the weight $k$ but it depends
only on $\chi_0$ and consequently only on $N$ (since there are only finitely many quadratic Dirichlet characters unramified away from $N$). Therefore, in this case, there exists prime $p$
with $\gcd(p, N)=1$ and $p \ll_N 1$ such that $\lambda(p) < 0$.

If $F$ is of Soudry type, then, for any prime $\gcd(p, N) = 1$,
$$
\lambda(p) = \left( p^{1/2}+p^{-1/2} \right) \lambda_0(p),
$$
where $\lambda_0(p)$ is as in \propref{p:lfunctionsfacts}. Therefore the
bound for the first sign change is the same as that of $\lambda_0(p)$.
Recall here that the $\lambda_0(p)$ are the Hecke eigenvalues
for a cusp form of $\GL_2$ of weight $1$ and level bounded by
a constant depending only on $N$. So by Theorem 1.1 of \cite{CK21}
there exists prime $p$ with $\gcd(p, N)=1$ and $p \ll_N 1$ such that
$\lambda_0(p) < 0$. This implies that $\lambda(p) < 0$.
%\fbox{issue with negative?}

Finally, we assume that $F$ is of Howe--Piatetski-Shapiro type. Hence
for any prime $\gcd(p, N)=1$ we have
$$
\lambda(p) = \left( p^{1/2}+p^{-1/2} \right)\left( \chi_1(p)+\chi_2(p) \right),
$$
where as in \propref{p:lfunctionsfacts}, the Dirichlet characters $\chi_1, \chi_2$ are distinct. Clearly the sign
of $\lambda(p)$ is the same as that of $\chi_1(p)+\chi_2(p)$ and for each $N$ there are only finitely many possibilities for $\chi_1(p)+\chi_2(p)$. It only remains to prove that given $\chi_1$ and $\chi_2$ that
$\chi_1(p)+\chi_2(p)$ changes sign infinitely often.
Indeed, we have the following more precise information.
Let $K_1, K_2$ be quadratic fields associated to the
quadratic Dirichlet characters $\chi_1, \chi_2$ respectively.
The Dirichlet characters $\chi_1 \ne \chi_2$ implies $K_1 \ne K_2$.  Then, for $\gcd(p,N)=1$ we have
\begin{comment}$$
\chi_i(p) =
\begin{cases}
1 & \text{if } p \text{ is split in } K_i,\\
-1 & \text{if } p \text{ is inert in } K_i,\\
0 & \text{if } p \text{ is ramified in } K_i.
\end{cases}
$$
Hence for $\gcd(p, N)=1$ we have\end{comment}
\begin{equation}\label{two-characters}
\chi_1(p)+\chi_2(p) ~=~
\begin{cases}
2 & \text{if } p \text{ is split in both } K_1 \text{ and } K_2,\\
-2 & \text{if } p \text{ is inert in } K_1 \text{ and } K_2,\\
0 & \text{if } p \text{ is split in one field and inert in other field}.
\end{cases}
\end{equation}
Therefore, $\chi_1(p)+\chi_2(p)$ changes sign infinitely often as $p$ runs over $\gcd(p, N)=1$.
\end{proof}

\begin{prop}\label{SK-HP}
Let $F \in S_k(\G_0^{(2)}(N))^{ \rm CAP}$ be a Hecke eigenform at all good primes with normalized
Hecke eigenvalues $\lambda(m)$ for $\gcd(m,N)=1$. Assume that $F$ does not lie in $S_k^*(\Gamma^{(2)}_0(N))$.  Then, for any
positive integers $a, M$ with $\gcd(aN, M)=1$, the sequence
$\{ \lambda(p) \}_{p \equiv a \mod M}$ changes sign
infinitely often.
\end{prop}

\begin{proof}
The proof for the various cases follow a similar pattern.
We shall give full details of the proof when $F$ is of non-classical
Saito--Kurokawa type and indicate the main changes
in the proof for the Howe--Piatetski-Shapiro and Soudry types.

Assume that $F \in S_k(\G_0^{(2)}(N))$ is of Saito--Kurokawa
type but not of classical Saito--Kurokawa type. Then,
by \propref{p:lfunctionsfacts}, there is a primitive quadratic
Dirichlet character $\chi_0$ such that
%and a representation $\pi_0$ of ${\rm GL}_2(\mathbb A)$ with normalized
%Hecke eigenvalues $\lambda_0(p)$ at primes such that
$$
\lambda(p) = \lambda_0(p) + \left( p^{1/2} + p^{-1/2}\right) \chi_0(p)
$$
for all primes $\gcd(p, N)=1$. Here $\lambda_0(p)$ is as in
\propref{p:lfunctionsfacts}. We note here that sign of
$\lambda(p)$ is same as that of $\chi_0(p)$ when $\gcd(p,N)=1$.
It remains to prove that $\chi_0(p)$ changes sign infinitely
often when $p$ runs over $p \equiv a \mod M$ for any $\gcd(aN, M)=1$.

Observe that, for $X$ sufficiently large, we have
\begin{eqnarray}\label{positive}
\sum_{p \le X, (p,N)=1 \atop {p\equiv a \mod M \atop \chi_0(p) = 1}} \frac{1}{p}
& = &
\frac{1}{2}\sum_{p \le X \atop (p,N)=1} \frac{\chi_0(p)+1}{p}
\times
\frac{1}{\varphi(M)}\sum_{\chi \mod M} \chi(a)^{-1}\chi(p)
\\
&=&
\frac{1}{2\varphi(M)}\sum_{\chi \mod M} \chi(a)^{-1}
\left( \sum_{p \le X \atop (p,N)=1} \frac{\chi_0(p)\chi(p)}{p}
+
\sum_{p \le X \atop (p,N)=1} \frac{\chi(p)}{p}\right). \nonumber
\end{eqnarray}
Note that $\chi_0\chi \ne 1$ as $\gcd(N, M)=1$. Hence
$$
\sum_{p \le X, (p,N)=1 \atop {p\equiv a \mod M \atop \chi_0(p) = 1}} \frac{1}{p}
=
\frac{1}{2\varphi(M)} \log\log X + O_{F, M}(1).
$$
Similarly, using $1 -\chi_0(p)$ in \eqref{positive} instead
of $\chi_0(p)+1$, one can show
$$
\sum_{p \le X, (p,N)=1 \atop {p\equiv a \mod M \atop \chi_0(p) = -1}} \frac{1}{p}
=
\frac{1}{2\varphi(M)} \log\log X + O_{F, M}(1).
$$
This completes the proof in the case when $F$ is a non-classical
Saito--Kurokawa lift.

When $F$ is of Howe--Piatetski-Shapiro type and $\chi_1, \chi_2$
are two distinct quadratic Dirichlet characters as in \propref{p:lfunctionsfacts}.
We need to show that $\chi_1(p)+\chi_2(p)$ takes infinitely
many positive and negative values as $p$ runs over
$p \equiv a \mod M$ when $\gcd(aN, M)=1$. From \eqref{two-characters},
it is clear that one can use
$$
\frac{1}{8}\big( \chi_1(p)+ \chi_2(p) + 2 \big)\big( \chi_1(p) + \chi_2(p)\big)
=
\frac{1}{4} \big( 1 + \chi_1(p)\chi_2(p) + \chi_1(p) + \chi_2(p) \big)
$$
and
$$
\frac{1}{8}\big( \chi_1(p) + \chi_2(p) - 2 \big)\big( \chi_1(p) + \chi_2(p)\big)
=
\frac{1}{4}\big(1+ \chi_1(p)\chi_2(p) - \chi_1(p) - \chi_2(p)\big)
$$
as characteristic functions for $\{(p, N)=1 ~|~ \chi_1(p)+\chi_2(p)>0\}$
and $\{(p, N)=1 ~|~ \chi_1(p)+\chi_2(p)<0\}$ respectively
and argue as above to complete the proof.

Finally, we consider the case when $F$ is of Soudry type. By \propref{p:lfunctionsfacts}, there is a unitary cuspidal automorphic
representation $\pi_0$ of ${\rm GL}_2(\mathbb A)$ such that
$$
\lambda(p) = \left( p^{1/2}+p^{-1/2} \right)\lambda_0(p)
$$
for any prime $\gcd(p, N)=1$. Here $\lambda_0(p)$ is the normalized
Hecke eigenvalue of $\pi_0$ at $p$. Clearly, it is enough to prove
the statement of the proposition for $\lambda_0(p)$ which can be
done in a similar manner to the proof of \propref{G-Y}. We leave the details for
the reader.
\end{proof}

\emph{Proofs of main results.}
\thmref{First-sign} follows from Propositions \ref{non-CAP} and \ref{CAP}. \thmref{split-inert} follows from Propositions \ref{G-Y} and \ref{SK-HP}.
\qed

\subsection{Odds and ends}
We explain some of the issues that can occur with Theorem \ref{split-inert} when $\gcd(N, M) >1$.
If the level $N$ is square free, then \thmref{split-inert} can be shown to hold for
any positive integers $a$ and $M$ such that $\gcd(a, M)=1$. In this case there is no need
to assume $\gcd(N, M)=1$. However for non-squarefree levels $N$ there are some subtle issues.

For an immediate example of the kind of issues that can occur, consider the case when $F$ is of Yoshida type which is coming from two CM forms
$f_1$ of weight $2$ and $f_2$ of weight $2k-2$ and assume that they both have CM by the same imaginary quadratic field $K$.
Then $\lambda(p) = \lambda_1(p)+\lambda_2(p) = 0$
for all primes $p$ which are inert in $K$. Hence there is an
arithmetic progression $a \mod M$ where $\gcd(M, N)>1$ such that
$\lambda(p) = 0$ for all $p \equiv a \mod M$.

Furthermore, for the CAP forms of Saito-Kurokawa or Howe--Piatetski-Shapiro type, one can construct arithmetic progressions (with the modulus $M$ only divisible by primes dividing the level) such that all Hecke eigenvalues at primes in the progression are positive.
\begin{lem}Let $F \in S_k(\G_0^{(2)}(N))$ be a Hecke eigenform with
normalized eigenvalues $\lambda(m)$.
Also assume that $F$ is either of Saito--Kurokawa type or of
Howe--Piatetski-Shapiro type. Then there exists a positive
integer $M$ which is only divisible by primes dividng $N$ such that $\lambda(p) > 0$ for
all $p \equiv 1 \mod M$.
\end{lem}
\begin{proof}
Let us first assume that $F \in S_k(\G_0^{(2)}(N))$ is
of Saito--Kurokawa type. Then, by \propref{p:lfunctionsfacts},
there is a primitive Dirichlet character $\chi_0  \mod M$ with $\chi^2 = 1$ and a
representation $\pi_0$ of ${\rm GL}_2(\mathbb A)$ with normalized
Hecke eigenvalues $\lambda_0(p)$ at primes such that
$$
\lambda(p) = \lambda_0(p) + p^{1/2}\chi_0(p) + p^{-1/2}\chi_0(p)
$$
for all primes $\gcd(p, N)=1$. Note that if $p \equiv 1 \mod M$ then
$\chi_0(p)=1$ and hence $\lambda(p) = \lambda_0(p)+p^{1/2}+p^{-1/2} > 0$
for all $p \equiv 1 \mod MN$.

Let $F$ be of Howe--Piatetski-Shapiro type. \propref{p:lfunctionsfacts} implies that there
are two distinct quadratic Dirichlet characters $\chi_1, \chi_2$ with
conductors $M_1, M_2$ respectively such that
$$
\lambda(p) = \left( p^{1/2}+p^{-1/2} \right)\chi_1(p) + \left( p^{1/2}+p^{-1/2}\right)\chi_2(p)
$$
for all $\gcd(p, N)=1$. Note that if $p \equiv 1 \mod M_1M_2$ then $\chi_1(p)=1=\chi_2(p)$.
Hence $\lambda(p) = 2(p^{1/2}+p^{-1/2}) > 0$ for all $p \equiv 1 \mod M_1M_2N$.
\end{proof}

\section{Sign changes of radial Fourier coefficients of cusp forms}
The main aim of this section is to prove Theorem \ref{t:fcsignchange}.
We divide the proof into a few parts.

\subsection{Relations between Fourier coefficients}
We will need the following proposition.

\begin{prop}\label{F-H relation}
Let $F \in  S_k(\G_0^{(2)}(N))$ be a Hecke eigenform at all good
primes having Fourier coefficients $a(T) ~( T \in \n_2)$.
Then, for a fundamental matrix $T_0 \in \n_2$ and a prime $p$ with
$\gcd(p,N)=1$ such that  $p$ is inert in $K = \Q(\sqrt{ -4~{\rm det }~ T_0})$,
we have
$$
a(pT_0) = a(T_0) \eta(p)
$$
where $\eta(p)$ is the $p$-th Hecke eigenvalue of $F$.
\end{prop}

\begin{proof}
This can be derived from a fundamental identity of Andrianov
\cite[Theorem 5.28]{AA09}. However we give here
another proof relying on Bessel models.
Let $d := -4 \ {\rm det}~ T_0$, and $\Lambda$ be a character
of the class group $\Cl_d$ of $K:=\Q(\sqrt{d})$. It is well-known
that (see, e.g., \cite[p.84]{AA74} or \cite[Prop 5.3]{pssmb}) that $\Cl_d$ can be identified
with the set of $\SL_2(\Z)$-equivalence classes of matrices
$S \in \n_2$ with $-4 \ {\rm det}~S = d$. Let $\{ N_i \}_{i = 1}^h$
be a complete set of representatives of $\Cl_d$ with $N_1 = T_0$.
By a special case of Corollary 2 of \cite{marzec21} (which
generalizes Theorem 2.10 of \cite{KST12}) we obtain for each
$\gcd(p,N)=1$ the relation
\begin{equation} \label{e:fourierrelation}
\sum_{i=1}^h
\Lambda(N_i)a(pN_i) =
p^{k-\frac{3}{2}}U^{1, 0}_{p}(\alpha_{p,1},\alpha_{p,3}; K, \Lambda)\sum_{i=1}^h \Lambda(N_i)a(N_i),
\end{equation}
where $\alpha_{p,i}$ are the Satake parameters and
$U^{\ell, m}_{p}( -, - ; K, \Lambda)$ is the (normalized) Bessel
function attached to a spherical representation, as defined in
\cite[Theorem 2.5]{KST12}. Using Sugano's formula for $U^{1, 0}_{p}$
as given in \cite[(2.3.6)]{KST12}, we have that
\begin{equation}\label{e:sugano}
U^{1, 0}_{p}(a,b; K, \Lambda) = a+b+a^{-1} + b^{-1} - p^{-1/2}\gamma,
\end{equation}
where for each $p$ coprime to $N$ we have
\begin{equation}
\gamma =
\begin{cases}
0 &\text{ if } p \text{ is inert in } K, \\
\Lambda_p(\varpi_{K,p}) &\text{ if } p \text{ is ramified in } K, \\
\Lambda_p(\varpi_{K,p}) + \overline{\Lambda_p(\varpi_{K,p})}
&
\text{ if } p \text{ is split in } K.
\end{cases}
\end{equation}
By our assumption $p$ is inert in $K$, so we have $\gamma=0$
and thus obtain $p^{k-\frac{3}{2}}(U^{1, 0}_{p}(\alpha_{p,1},\alpha_{p,3}))
=
p^{k-\frac{3}{2}}(\alpha_{p,1} + \alpha_{p,3} + \alpha_{p,1}^{-1} + \alpha_{p,3}^{-1})
=
\eta(p)$
and so \eqref{e:fourierrelation} reduces to
$$
\sum_{i=1}^h
\Lambda(N_i)a(pN_i) =
\eta(p)\sum_{i=1}^h \Lambda(N_i)a(N_i).
$$
Since this is true for any character $\Lambda$ of $\Cl_d$, we
use orthogonality of characters to deduce that
$a(pN_i) = \eta(p)a(N_i)$ for each $i$.
\end{proof}
Applying \thmref{split-inert} and \propref{F-H relation}, it is
easy to obtain the following weaker form of Theorem \ref{t:fcsignchange}.

\begin{cor}\label{F-Hecke}
Let $F \in S_k(\G_0^{(2)}(N))$ be a Hecke eigenform at all good
primes having Fourier coefficients $a(T)$. Assume that $F$ does
not lie in $S_k^*(\G_0^{(2)}(N))$. If $k=2$ and $F$ is not attached to a CAP representation, then assume further that $F$ satisfies the Ramanujan conjecture.  Then for any fundamental
$T_0 \in \n_2$ such that $\gcd(4\det(T_0), N)=1$ and
$a(T_0) \ne 0$,  there are infinitely many primes $p$ with
$a(pT_0) > 0$ and there are infinitely many primes $p$ such
that $a(pT_0) < 0$.
\end{cor}

\begin{proof}
Let $\eta(p)$ be the $p$-th Hecke eigenvalue of $F$. Then for any
fundamental $T_0$ with $a(T_0) \ne 0$, \propref{F-H relation} implies
$$
a(pT_0) = a(T_0) \eta(p)
$$
for any prime $p$ which is coprime to $N$ and inert in
$K = \Q(\sqrt{-4{\rm det}~T_0})$. Since the latter set
contains an arithmetic progression modulo $M=4\det(T_0)$,
we apply \thmref{split-inert} to complete the proof.
\end{proof}
While \corref{F-Hecke} is an
easy consequence of \thmref{split-inert}
and \propref{F-H relation}, we will prove below in
Proposition \ref{F-o-Maass} that one can extend
\corref{F-Hecke} to an \emph{arbitrary} cusp form
$F \in S_k^*(\G_0^{(2)}(N))^\perp$. Note here that if $F \in S_k(\G_0^{(2)}(N))$ is a Hecke eigenform at all
good primes that does not lie in $S_k^*(\G_0^{(2)}(N))$ then
$F$ lies
in $S_k^*(\G_0^{(2)}(N))^\perp$.

The next remark clarifies what is known about the existence
of fundamental matrices $T_0$ as above. In particular, if $N = 1$
(or if one makes some mild assumptions
on $F$ at the primes dividing  odd and squarefree $N$) then there are lots of
$T_0 \in \n_2$ such that $\gcd(4\det(T_0), N)=1$ and
$a(T_0) \ne 0$.

\begin{rmk}\label{remark:nonvanishing}
Suppose that $N$ is odd and squarefree and $F \in S_k(\Gamma^{(2)}_0(N))$ is non-zero with
Fourier coefficients $a(T)$. If $N>1$, assume that $F$ is an
eigenform for the $U(p)$ Hecke operator for the finitely many
primes $p|N$; we make no assumptions concerning whether
$F$ is a Hecke eigenform at primes not dividing the level $N$.
Then by the results of \cite{JLS21+}, one knows that given
$\epsilon>0$, for all sufficiently large $X$ there are
$\ge X^{1-\epsilon}$ distinct odd squarefree integers
$n_i \in [X, 2X]$ satisfying $\gcd(n_i, N)=1$ and $-n_i \equiv 1 \mod 4$
(in particular $-n_i$ is a fundamental discriminant) such that
for each $n_i$ as above there is a fundamental matrix $T_i$
with $4{\rm det}~T_i = n_i$ and $a(T_i) \neq 0$.
\end{rmk}

%\fbox{move next result earlier}
\begin{comment}
\begin{rmk}\label{oM-rmk}
Let $F \in S_k^*(\G_0^{(2)}(N))^\perp$ be a Hecke eigenform
with Fourier coefficient $a(T)$.
Then, for any fundamental $T_0 \in \n_2$, using \eqref{alpha} and \eqref{F-p}
and Weissauer bound $|\eta(p)| \le 4p^{k-3/2}$ one can derive
\end{rmk}
\end{comment}

We will also need another relation which is specific to
the classical Saito--Kurokawa lifts.

\begin{lem}\label{Kohnen}
Let $F \in S_k^*(\G_0^{(2)}(N))$ be a Hecke eigenform
at all good primes with normalized Hecke eigenvalues
$\lambda(m)$ for $\gcd(m,N)=1$. Let $\pi_0$ be the
representation of $\GL_2(\A)$ associated to $F$ (see case
\ref{case:SK} of Proposition \ref{p:lfunctionsfacts}) and
let $\lambda_0(n)$ denote the normalized Hecke eigenvalue
of $\pi_0$ at $n$.
Let $a(T)~(T \in \n_2)$ be the Fourier coefficients of $F$ and
let $T_0$ be a fundamental matrix with the fundamental
discriminant $d := -4{\rm det}~T_0$. Then
$$
a(pT_0)
~=~
a(T_0) p^{k-1} \left[ 1 - \left( \frac{d}{p} \right)\frac{1}{p}
~+~
\frac{\lambda_0(p)}{p^{1/2}} \right]
$$
for any prime $p$.
\end{lem}

\begin{proof}
Let $\{ N_i \}_{i = 1}^h$ be as in the proof of Proposition \ref{F-H relation}.
By a special case of Lemma 9 of \cite{marzec21} (see also
Theorem 5.1 of \cite{pssmb}) $a(N_i) = a(N_j)$ and $a(pN_i) = a(pN_j)$
for $1 \le i, j \le h$. So applying \eqref{e:fourierrelation} with $\Lambda=1$,
we obtain
$$
h a(pT_0)
~=~
h a(T_0) \left(p^{k-\frac{3}{2}}\lambda(p) - p^{k-2}(1+\left( \frac{d}{p} \right))\right).
$$
Since $\lambda(p) = \lambda_0(p) + p^{1/2}+p^{-1/2}$
by case \ref{case:SK} of Proposition \ref{p:lfunctionsfacts},
we obtain the desired result.
\end{proof}

\subsection{A Selberg orthogonality relation for linear combinations of Hecke eigenforms}
The object of this subsection is to prove the following
proposition which may be viewed as an extension of
Lemma \ref{inert-prop} to the case of linear combinations
of eigenforms in $S_k^*(\G_0^{(2)}(N))^\perp$ (including
combinations from different types of Arthur packets).

\begin{prop}\label{prop:selberggen}Let $k$ and $N$ be positive integers. Let $F_1, F_2, \ldots, F_m$ be elements of $S_k^*(\G_0^{(2)}(N))^\perp$ such that each $F_i$ is a Hecke eigenform at all good primes with normalized Hecke eigenvalues $\lambda_i(n)$ for $\gcd(n,N)=1$. Let $r_1, r_2, \ldots, r_m$ be complex numbers such that $\sum_{i=1}^m r_i \neq 0$. Denote $$R = (r_1, r_2, \ldots, r_m; F_1, F_2, \ldots, F_m).$$  For each prime $p$ coprime to $N$, define \[a_R(p) = \sum_{i=1}^m r_i\lambda_i(p).\] Then  exactly one of the following two cases must occur (depending on $R$).

\begin{enumerate}
\item \textbf{Case 1.}  For $a, M$ positive
integers with $\gcd(aN, M) = 1$ and for all sufficiently large $X$, we have
$$
\sum_{p \le X, (p, N)=1 \atop p \equiv a \mod M} \frac{a_R(p)}{p}
~=~ O_{R, M}(1)
\phantom{m}
$$
and
$$
\sum_{p \le X, (p, N)=1 \atop p \equiv a \mod M} \frac{|a_R(p)|^2}{p}
~=~
C_{R,M} \log\log X ~+~ O_{R, M}(1) $$ for some positive real $C_{R,M}$ depending on $R$ and $M$.

Moreover, in this case, under the additional assumption (needed only for $k=2$) that each $F_i$ above that is of general type satisfies the Ramanujan conjecture, we have $|a_R(p)| \ll_R 1$ for all primes $p$.

\item \textbf{Case 2.}  For $a, M$ positive
integers with $\gcd(aN, M) = 1$ and for all sufficiently large $X$, we have
$$
\sum_{p \le X, (p, N)=1 \atop p \equiv a \mod M} \frac{a_R(p)}{p^{3/2}}
~=~ O_{R, M}(1)
\phantom{m}
$$
and
$$
\sum_{p \le X, (p, N)=1 \atop p \equiv a \mod M} \frac{|a_R(p)|^2}{p^2}
~=~
C_{R,M} \log\log X ~+~ O_{R, M}(1) $$ for some positive real $C_{R,M}$ depending on $R$ and $M$.

Moreover, in this case, $|a_R(p)| \ll_R p^{1/2}$ for all primes $p$.
\end{enumerate}
\end{prop}
\begin{proof}
According to Proposition \ref{p:lfunctionsfacts}, a Hecke eigenform $F \in S_k^*(\G_0^{(2)}(N))^\perp$ at all good primes gives rise  to a certain finite set $\Sigma(F)$ of unitary cuspidal automorphic self-dual irreducible representations of $\GL_1$, $\GL_2$ or $\GL_4$ such that the Hecke eigenvalues of $F$ are related to those of the representations in $\Sigma(F)$ as follows:

\begin{itemize}

\item If $F$ is of general type, then we take $\Sigma(F) = \{\Pi\}$ where $\Pi$ is as in Case 1 of Proposition \ref{p:lfunctionsfacts}. In this case we trivially have $\lambda_F(p) = \lambda_\Pi(p)$.

\item If $F$ is of Yoshida type, then we take $\Sigma(F) = \{\pi_1, \pi_2\}$ where $\pi_1, \pi_2$ are as in Case 2 of Proposition \ref{p:lfunctionsfacts}. In this case we have $\lambda_F(p) =  \lambda_{\pi_1}(p) + \lambda_{\pi_2}(p)$.

\item If $F$ is of Saito-Kurokawa type, then we take $\Sigma(F) = \{\pi_0, \chi_0\}$ where $\pi_0, \chi_0$ are as in Case 3 of Proposition \ref{p:lfunctionsfacts}. In this case we have $\lambda_F(p) =  \lambda_{\pi_0}(p) + (p^{1/2} + p^{-1/2} )\lambda_{\chi_0}(p)$.

\item If $F$ is of Soudry type, then we take $\Sigma(F) = \{\pi_0\}$ where $\pi_0$ is as in Case 4 of Proposition \ref{p:lfunctionsfacts}. In this case we have $\lambda_F(p) =  (p^{1/2} + p^{-1/2} ) \lambda_{\pi_0}(p)$.

\item If $F$ is of Howe--Piatetski-Shapiro type, then we take $\Sigma(F) = \{\chi_1, \chi_2\}$ where $\chi_1, \chi_2$ are as in Case 5 of Proposition \ref{p:lfunctionsfacts}. In this case we have $\lambda_F(p) =  (p^{1/2} + p^{-1/2} ) \left(\lambda_{\chi_1}(p) + \lambda_{\chi_2}(p) \right)$.

\end{itemize}

 Let $\{\pi_j\}_{1\le j \le n}$ denote the set of all unitary cuspidal automorphic self-dual irreducible representations of $\GL_1$, $\GL_2$ or $\GL_4$ that occur among the elements of $\Sigma(F)$ as $F\in S_k^*(\G_0^{(2)}(N))^\perp$ varies over Hecke eigenforms at all good primes. Note that none of the $\pi_j$ equals the trivial representation. Let $\mu_j(p)$ denote the normalized Hecke eigenvalue of $\pi_j$ at a prime $p$. For $1\le i \le m$, $1 \le j \le n$ and primes $p \nmid N$ we set $c_i^{(j)}(p) = 0$ if $\pi_j \notin \Sigma(F_i)$ and when $\pi_j \in \Sigma(F_i)$ we set $c_i^{(j)}(p)$ equal to either 1 or $(p^{1/2} + p^{-1/2})$ (depending on the coefficient of the Hecke eigenvalue as explicated above) so that for all $1\le i \le m$ and all $p \nmid N$ we have \begin{equation}\label{e:deflambdai}\lambda_i(p) = \sum_{j=1}^n c_i^{(j)}(p) \mu_j(p).\end{equation}

\begin{comment}
\item For each fixed $i,j$, as a function of $p$, $c_i^{(j)}(p)$ equals one of the following: the constant function 0, the constant function 1, or the function $p \mapsto (p^{1/2} + p^{-1/2})$.

 \item
If $\pi_j$ is either a representation of $\GL_4(\A)$ or a representation of $\GL_2(\A)$ corrresponding to a classical newform of weight $\ge 2$, then $c_i^{(j)}(p) \in \{0, 1\}.$

 \item
If $\pi_j$ is either a representation of $\GL_2(\A)$ corrresponding to a classical newform of weight $1$, or a representation (character) of $\GL_1(\A)$, then $c_i^{(j)}(p) \in \{0, p^{1/2} + p^{-1/2}\}.$
\end{enumerate}
\end{comment}

For $1\le j \le n$, set
\begin{equation}\label{e:defRj}
R^{(j)}(p) = \sum_{i=1}^m r_i c_i^{(j)}(p).
\end{equation}
This shows that
\begin{equation}\label{e:rjform}R^{(j)}(p) = D_j + E_j(p^{1/2}+p^{-1/2}),\end{equation} where $D_j, E_j$ are
complex numbers and independent of $p$.

Using \eqref{e:deflambdai}, \eqref{e:defRj} and the definition of $a_R(p)$, we see that

\begin{equation}\label{e:AR1}
a_R(p) =  \sum_{j=1}^n R^{(j)}(p) \mu_j(p)
\end{equation}
\begin{equation}\label{e:AR2}
|a_R(p)|^2 = \sum_{j=1}^n |R^{(j)}(p)|^2 \mu_j(p)^2
+
\sum_{1\le j_1 \neq j_2 \le n} R^{(j_1)}(p)\overline{R^{(j_2)}(p)} \mu_{j_1}(p)\mu_{j_2}(p).
\end{equation}

We now claim that the function $ p \mapsto R^{(j)}(p)$ cannot be
identically 0 for all $j$.
Indeed, set $C$ to equal the $m \times n$ matrix $(c_i^{(j)}(p))_{i,j}$
and let $V = (v_j)_{1\le j \le n}$ equal the column matrix  defined by
$$
v_j = \begin{cases}
1 & \text{ if } \pi_j \text{ is a $\GL_4$ representation } \\
1/2 & \text{ if } \pi_j \text{ is a $\GL_2$ representation of weight $\ge $ 2 } \\
\frac{1}{p^{1/2} + p^{-1/2}} & \text{ if } \pi_j \text{ is a $\GL_2$ representation of weight  1 } \\
\frac{1}{2(p^{1/2} + p^{-1/2})} & \text{ if } \pi_j \text{ is a $\GL_1$ character. }
\end{cases}
$$
Then it is easy to check that $$C V = (1,1, \ldots, 1)^t.$$
Consequently, we have \begin{align*}(R^{(1)}(p), \ R^{(2)}(p), \ \ldots, R^{(n)}(p)) \ V & = (r_1, \ r_2, \ \ldots, r_m) CV \\&= r_1 + r_2 + \ldots r_m \\& \neq 0\end{align*} thus showing that the claim is true.
Now that we have proved that $R^{(j)}(p)$ cannot be identically 0, using the definition of $R^{(j)}(p)$ we see that there are two cases.

In the first case (which corresponds to Case 1 of the statement of the proposition) none of the $R^{(j)}(p)$ actually depend on $p$; i.e, in the notation of \eqref{e:rjform}, each $E_j=0$.  In this case, $R^{(j)}(p) \ll_R 1$ for all $1\le j \le n$ and $p$, and furthermore there exists some $j_0$ such that $R^{(j_0)}(p) \neq 0$. Therefore using \eqref{e:AR1}, \eqref{e:AR2} and Corollary \ref{cor:selberg}, we conclude the proof. Note that if each $\pi_j$ satisfies the Ramanujan conjecture, then we have $|\mu_j(p)| \ll 1$ which gives the required bound on $a_R(p)$ from \eqref{e:AR2}.

In the second case (which corresponds to Case 2 of the statement of the proposition) at least one of the $R^{(j)}(p)$ depends on $p$. In this case, using \eqref{e:rjform}, we see that  $R^{(j)}(p) \ll_R p^{1/2}$ for all $1\le j \le n$, and furthermore there exists some $j_0$ such that $R^{(j_0)}(p) \asymp_R p^{1/2}$.  Again using \eqref{e:AR1}, \eqref{e:AR2} and Corollary \ref{cor:selberg}, we conclude the proof.
\end{proof}

\subsection{Sign changes for forms orthogonal to classical Saito-Kurokawa lifts}

\begin{prop}\label{F-o-Maass}
Let $k$ and $N$ be positive integers. Let $F \in S_k^*(\G_0^{(2)}(N))^\perp$ be a non-zero cusp form having
real Fourier coefficients $a(T)$ and suppose that $a(T_0) \ne 0$ for some
fundamental $T_0 \in \n_2$ with $\gcd(4\det(T_0), N)=1$. If $k=2$, assume that each element of $S_k(\G_0^{(2)}(N))$ that is of general type satisfies the Ramanujan conjecture. Then the
sequence $\{ a(pT_0) \}_{p \text{ inert in } \Q(\sqrt{-4{\rm det}~T_0})}$
changes sign infinitely often.
\end{prop}

\begin{proof}
Let us write $F$ as a finite sum
$$
F = \sum_{i=1}^m F_i,
$$
where  $F_i$'s are Hecke eigenforms at all good places having normalized Hecke eigenvalues
$\lambda_i(m)$ (for $m \ge 1, (m,N)=1$) and Fourier coefficients $a_i(T) ~ (T \in \n_2)$.
Let us fix a fundamental $T_0 \in \n_2$ such that $a(T_0) \neq 0$ and put $K := \Q(\sqrt{-4~{\rm det}~ T_0})$
and let $P_{K, -}$ denote the set of primes inert in $K$. Put $M:=4\det(T_0)$. Fix some $a$ such that  $\gcd(aN, M)=1$ and such that each prime $p \equiv a \mod M$ lies in  $P_{K, -}$. It suffices to show that the sequence $\{ a(pT_0) \}_{p \equiv a \mod M}$
changes sign infinitely often.

Put $r_i = a_i(T_0)$. So $\sum_{i=1}^m r_i = a(T_0)\neq 0$. Following the  notation of Proposition \ref{prop:selberggen}, define $a_R(p) = \sum_{i=1}^m r_i\lambda_i(p)$ and let $R = (r_1, r_2, \ldots, r_m; F_1, F_2, \ldots, F_m)$; note that $R$ depends only on $F$ and $T_0$.

Using Proposition \ref{F-H relation} we have for each $p \equiv a \mod M$,
\[\frac{a(pT_0)}{p^{k-1/2}} = \frac{a_R(p)}{p}, \quad \frac{|a(pT_0)|^2}{p^{2k-2}} = \frac{|a_R(p)|^2}{p}\]

We now consider two cases according to which case of Proposition \ref{prop:selberggen} we find ourselves in.

In \textbf{Case 1}, there is a constant  $C_1$ depending on $F$ and $T_0$ such that for each $p \equiv a \mod M$ $$
\frac{|a(pT_0)|^2}{p^{2k-2}} = \frac{|a_R(p)|^2}{p}
~\le~
C_1
\frac{|a_R(p)|}{p} = C_1 \left|\frac{a(pT_0)}{p^{k-1/2}}\right|,
$$ and for all large $X$ using Proposition \ref{prop:selberggen} we have
\begin{equation}\label{F-one}
\sum_{p \le X \atop p\equiv a \mod M} \frac{a(pT_0)}{p^{k-1/2}}
~=~
\sum_{p \le X \atop p\equiv a \mod M} \frac{a_R(p)}{p}
~\ll_{F, T_0}~ 1,
\end{equation}

\begin{eqnarray}\label{F-two}
\sum_{p \le X \atop p\equiv a \mod M} \frac{|a(pT_0)|^2}{p^{2k-2}}
~=~
\sum_{p \le X \atop p\equiv a \mod M} \frac{|a_R(p)|^2}{p}
= C_{F, T_0} \log \log X + O_{F,T_0}(1).
\end{eqnarray}
Now for any $0 < \delta < 1$, let $y = X^\delta$.
Assume that $a(pT_0)$ has the same sign $\delta \in \{1, -1\}$ for any
prime $ y \le p \le X$ and $p \equiv a \mod M$.
Therefore \eqref{F-one} and \eqref{F-two} imply
$$
C_{F, T_0} \log\left( \frac{\log X}{\log y} \right) ~+~ O_{F, T_0}(1)
~=~
\sum_{y \le p \le X \atop p\equiv a \mod M} \frac{|a(pT_0)|^2}{p^{2k-2}}
~\le~
C_1 \delta \sum_{y \le p \le X \atop  p\equiv a \mod M} \frac{a(pT_0)}{p^{k-1/2}}
~=~
O_{F,T_0}(1).
$$
This is absurd when $\delta > 0$ is sufficiently small. Hence we are done.

In \textbf{Case 2}, for each $p \equiv a \mod M$ there is a constant  $C_1$ such that $$
\frac{|a(pT_0)|^2}{p^{2k-1}} = \frac{|a_R(p)|^2}{p^2}
~\le~
C_1
\frac{|a_R(p)|}{p^{3/2}} = C_1 \left|\frac{a(pT_0)}{p^{k}}\right|,
$$ and for all large $X$ using Proposition \ref{prop:selberggen} we have
\begin{equation}\label{F-one2}
\sum_{p \le X \atop p\equiv a \mod M} \frac{a(pT_0)}{p^{k}}
~=~
\sum_{p \le X \atop p\equiv a \mod M} \frac{a_R(p)}{p^{3/2}}
~\ll_{F, T_0}~ 1,
\end{equation}

\begin{eqnarray}\label{F-two2}
\sum_{p \le X \atop p\equiv a \mod M} \frac{|a(pT_0)|^2}{p^{2k-1}}
~=~
\sum_{p \le X \atop p\equiv a \mod M} \frac{|a_R(p)|^2}{p^2}
= C_{F, T_0} \log \log X + O_{F,T_0}(1).
\end{eqnarray}
Now for any $0 < \delta < 1$, let $y = X^\delta$.
Assume that $a(pT_0)$ has the same sign $\delta \in \{1, -1\}$ for any
prime $ y \le p \le X$ and $p \equiv a \mod M$.
Therefore \eqref{F-one2} and \eqref{F-two2} imply
$$
C_{F, T_0} \log\left( \frac{\log X}{\log y} \right) ~+~ O_{F, T_0}(1)
~=~
\sum_{y \le p \le X \atop p\equiv a \mod M} \frac{|a(pT_0)|^2}{p^{2k-1}}
~\le~
C_1 \delta \sum_{y \le p \le X \atop  p\equiv a \mod M} \frac{a(pT_0)}{p^{k}}
~=~
O_{F,T_0}(1).
$$
This is absurd when $\delta > 0$ is sufficiently small. Hence we are done.
\end{proof}

\subsection{Sign changes for forms in the generalized Maass subspace}
The next result together with
\propref{F-o-Maass} completes the proof of Theorem \ref{t:fcsignchange}.

\begin{prop}\label{Maass}
Let $F \in S_k^*(\G_0^{(2)}(N))$ be a non-zero cusp form having real
Fourier coefficients $a(T)$ and $a(T_0) \ne 0$ for some fundamental
$T_0$. Then there exists a constant $C > 0$ such that the quantities
$a(pT_0)$ have the same sign for all primes $p > C$.
\end{prop}

\noindent
{\bf Proof of \propref{Maass}.}
Since $F \in  S_k^*(\G_0^{(2)}(N))$, there exist a set of
Hecke eigenforms (at all good places)
$F_1, \dots, F_\ell \in S_k^*(\G_0^{(2)}(N))$ so that
\begin{equation*}
F = \sum_{i = 1}^\ell  F_i.
\end{equation*}
We let $a_i(T)$ denote the Fourier coefficient of $F_i$. Then for any fundamental $T_0 \in \n_2$
with $a(T_0) \ne 0$ and for any prime $p$ \lemref{Kohnen} gives
$$
a_i(pT_0)
~=~ a_i(T_0)p^{k-1}
\left[ 1 - \left( \frac{d}{p} \right)\frac{1}{p}
~+~
\frac{\lambda_{0,i}(p)}{p^{1/2}} \right],
$$
where $d := -4{\rm det}~T_0$. Hence we have
\begin{eqnarray*}
a(pT_0)
& = &
\sum_{i = 1}^\ell a_i(pT_0)
\\
& = &
\sum_{i = 1}^\ell  a_i(T_0) p^{k-1}
\left[ 1 - \left( \frac{d}{p} \right)\frac{1}{p}
~+~
\frac{\lambda_{0,i}(p)}{p^{1/2}} \right]
\\
& = &
a(T_0) p^{k-1} ~+~ O_{F, T_0}\left( p^{k-3/2} \right).
\end{eqnarray*}
Here in the last line we used $\lambda_{0,i}(p) \ll 1$. The above
shows that for all large $p$, $a(pT_0)$ has the same sign as
$a(T_0)$ and this completes the proof.
\qed

\section{Bounds for Fourier coefficients of cusp forms in the Maass subspace}
In this section, we will prove \thmref{DK-SK}. We give two proofs, using very different methods.
For the first proof, carried out in Section \ref{S:SKbessel}, we build upon some calculations due to Pitale--Schmidt \cite{ps14} to obtain key bounds (see Proposition \ref{p:SKlocal}) quantifying the growth of the local Bessel functions  associated  to certain fixed vectors in representations of $\GSp_4(\Q_p)$. This relates the sizes of general Fourier coefficients of forms of Saito--Kurokawa type to the sizes of fundamental Fourier coefficients, and in the latter case the desired bounds easily follow from bounds for central values of $L$-functions. Using the above method, we are able to prove a refined form of Theorem \ref{DK-SK} in Theorem \ref{t:SKRS} below.

For the second proof, which is classical and carried out in Section \ref{s:SKClass}, we restrict ourselves to Saito-Kurokawa lifts of newforms of squarefree levels; see the statement of Proposition \ref{p:SKClass}. To prove Proposition \ref{p:SKClass}, we  follow a technique
of Kohnen \cite{WK92, WK93} together with recent developments in the
theory of Fourier--Jacobi expansion.
\subsection{Bounds for Fourier coefficients using local Bessel functions}\label{S:SKbessel}

For $S = \mat{a}{b/2}{b/2}{c}\in \n_2$ we denote $L_S:=\gcd(a,b,c)$, $\disc(S):=-4\det(S) = b^2 - 4ac$; we call $L_S$ the content of $S$ and call $\disc(S)$ the discriminant of $S$. For any $S = \mat{a}{b/2}{b/2}{c}\in \n_2$,  the non-zero integer $\disc(S)$ factors as follows:
\[\disc(S) = d_SL_S^2M_S^2\] where $d_S<0$ is a fundamental discriminant, and $M_S$ is a positive integer. Note that $S$ is fundamental if and only if $L_S=M_S=1$ in the above factorization. Moreover, any $S \in \n_2$ is $\SL_2(\Z)$-equivalent to a matrix \begin{equation}\label{e:nonfunddecomp}S' :=\mat{L_S}{}{}{L_S}\mat{M_S}{}{}{1}S_c\mat{M_S}{}{}{1}
\end{equation}
where $S_c$ is a fundamental matrix of discriminant $d_S$; see Proposition 5.3 of \cite{pssmb}.

It is clear that the fundamental discriminant $d_S$ is uniquely determined by $\disc(S)$, and therefore the pair of integers $(\disc(S), L_S)$ determine $d_S$ and $M_S$. A key distinguishing property of forms $F$ in $S_k^*(\G_0^{(2)}(N))$ --- which does not hold in general for $F \notin S_k^*(\G_0^{(2)}(N))$ --- is that a Fourier coefficient $a(S)$ of $F$ only depends on $\disc(S)$ and $L_S$ (or equivalently, only on $d_S$, $L_S$ and $M_S$).

\begin{lem}\label{l:SKkey}Let $F \in S_k^*(\G_0^{(2)}(N))$ be a non-zero cusp form with Fourier coefficients $a(T)$. Suppose that $S_1, S_2 \in \n_2$ with $\disc(S_1) = \disc(S_2)$ and $L_{S_1}=L_{S_2}$. Then $a(S_1)=a(S_2)$.
\end{lem}
\begin{proof}We may assume without loss of generality that $F$ gives rise to an irreducible automorphic representation, and thus is a Hecke eigenform at all good primes. The lemma was proved in the case $N=1$ in this setting in \cite[Theorem 5.1 (i)]{pssmb}, and the proof for general $N$ is essentially identical and follows from \cite[Lemma 9]{marzec21} and \eqref{e:nonfunddecomp} above.
\end{proof}

In the next lemma we collect some facts about the local representations associated to forms of classical Saito-Kurokawa type.

\begin{lem}\label{l:SK properties}Let $F \in S_k^*(\G_0^{(2)}(N))$ be a non-zero form in the generalized Maass subspace such that $F$ is a Hecke eigenform at all good primes. Let $\pi \in \Pi(F)$ and write $\pi\simeq\otimes_v \pi_v$. Let $f_0$ be a cuspidal holomorphic newform of weight $2k-2$ for $\Gamma_0(N_0)$ (see Proposition \ref{p:lfunctionsfacts}) whose adelization generates the automorphic representation  $\pi_0$ of $\GL_2(\A)$ such that $L^N(s, \pi) = L^N(s, \pi_0) \zeta^N(s+1/2)  \zeta^N(s-1/2)$. Then the following hold:
\begin{enumerate}
\item Suppose that $p$ is a prime  such that $p \nmid N_0$.\footnote{Note here that by Proposition \ref{p:lfunctionsfacts}, if $p|N_0$ then $p|N$; so the set of $p \nmid N$ is contained in the set of $p \nmid N_0$.} Then $\pi_p\simeq \chi_p \mathbf{1}_{\GL(2)} \rtimes \chi_p^{-1}$ is a representation  of type IIb in the notation of Table A.15 of \cite{NF}, where $\chi_p$ is an unramified unitary character of $\Q_p^\times$.

\item Suppose that $p$ is a prime such that $p | N_0$ and $v_p(N) = 1$. Then we have $v_p(N_0)=1$, and moreover either $\pi_p\simeq \tau(T, \nu^{-1/2})$ is  of type VIb in the notation of Table A.15 of \cite{NF}, or $\pi_p\simeq L((\nu^{1/2}\xi_p\mathrm{St}_{\GL(2)}, \nu^{-1/2})$ is of type Vb in the notation of Table A.15 of \cite{NF}, where $\xi_p$ is the unramified non-trivial quadratic character of $\Q_p^\times$. The former case above occurs iff $\pi_{0,p}$ is the Steinberg representation, and the latter case occurs iff $\pi_{0,p}$ is the unramified quadratic twist of the Steinberg representation.

\end{enumerate}

Furthermore, if $N$ is squarefree, then $F$ gives rise to an irreducible automorphic representation. Moreover, if $N$ is squarefree then $F$ is a newform (in the sense of Section 3.2 of \cite{DPSS20}) if and only if $N_0=N$.
\end{lem}

\begin{proof}By definition, the global Arthur parameter of $\pi$ is equal to $(\pi_0 \boxtimes 1) \boxplus (1 \boxtimes \nu(2))$; see the notation and explanation in Section 1 of \cite{schcap} and note that the quadratic character $\sigma$ (which equals the character $\chi_0$ in the notation of part 3 of Proposition \ref{p:lfunctionsfacts})) is trivial in the present case. Now the possibilities for $\pi_p$ can be read off from Table 2 of \cite{schcap}. In particular, if $p$ is a prime such that such that $p \nmid N_0$, then $\pi_{0,p}$ is an unramified principal series of the form $\chi_p \times \chi_p^{-1}$, in which case Table 2 of \cite{schcap} asserts that $\pi_p\simeq \chi_p \mathbf{1}_{\GL(2)} \rtimes \chi_p^{-1}$.

Next suppose $p|N_0$. Hence $p|N$ and so $\pi_p$ has a vector fixed by $K_{0,p}(N)$, the group defined in \eqref{k0pdef}. Looking at Table 2 of \cite{schcap} we see that $\pi_p$ is one of the types Vb, Va*, VIb, VIc, XIb or XIa*, and of these only Vb and VIb has a vector  fixed by $K_{0,p}(N)$ when $v_p(N)=1$ (as can be read off from Table A.15 of \cite{NF}).   So $\pi_p$ is either of Type VIb or Type Vb. Again from Table 2 of \cite{schcap}, we see that in the former case $\pi_{0,p}$ is the Steinberg representation, and in the latter case  $\pi_{0,p}$ is the unramified quadratic twist of the Steinberg representation. Clearly in both cases, the conductor exponent of $\pi_{0,p}$ equals 1, and hence $v_p(N_0)=1$.

Finally, suppose that $N$ is squarefree. By the previous parts we have that $N_0$ is also squarefree and we also know that $N_0 | N$. Furthermore, we know exactly the local components $\pi_p$ at each place. It follows that any two representations in $\pi(F)$ are isomorphic, and consequently, $F$ gives rise to an irreducible automorphic representation. If  $N_0=N$, then for each prime $p|N$ (as well as for each prime $p \nmid N)$, $\pi_p$ has a  1-dimensional space of $K_{0,p}(N)$-fixed vectors as follows from the previous parts and  Table A.15 of \cite{NF}; consequently, $F$ is a newform. Conversely, if $F$ is a newform then $\pi_p$ cannot be of Type IIb at any $p|N$, because the Type IIb representation being spherical does not contain a local newvector with respect to $K_{0,p}(N)$ (see Theorem 2.3.1 of \cite{sch05}). Therefore by the previous parts, $N_0=N$.

This concludes the proof.
\end{proof}
\begin{rmk} If $N$ is not squarefree, then there are 4 further possibilities for $\pi_p$ beyond the Types IIb, Vb and VIb mentioned in Lemma \ref{l:SK properties}, namely the types  Va*, VIc, XIb and XIa*.
\end{rmk}

Given a prime $p$ and non-negative integers $\ell, m$, define the matrix \[h_p(\ell, m) = \left[\begin{smallmatrix} p^{\ell + 2m}&&&\\&p^{\ell+m}&&\\&&1&\\&&&p^m\end{smallmatrix}\right]\]
We also need the subgroup $P_{1,p} \subset \GSp_4(\Z_p)$ which is defined as $P_{1,p}:=K_{0,p}(p)$ where $K_{0,p}(N)$ is defined in \eqref{k0pdef}.

Consider an irreducible admissible representation $\pi$ of $\GSp_4(\Q_p)$ that admits a local Bessel model with respect to a standard additive character $\theta$ (see Section 2 of \cite{pssmb}) and the trivial character $\Lambda=\mathbf{1}$ of a quadratic extension $L$ of $\Q_p$ ($L$ equals either $\Q_p \oplus \Q_p$ or a quadratic field extension of $\Q_p$). We fix an isomorphism $w \mapsto B_w$ (which is uniquely determined up to multiples) of $\pi$ with its $(\mathbf{1}, \theta)$-Bessel model. This associates to each vector $w \in V_\pi$ a function $B_w$ on $\GSp_4(\Q_p)$  and we are interested in bounding $B_w(h_p(\ell, m))$ as a function of $\ell$ and $m$. The heart of our method is the next proposition, which quantifies the growth of the functions  $B_w(h_p(\ell, m))$  associated to suitable vectors $w$ in local representations of Saito-Kurokawa type.

\begin{prop}\label{p:SKlocal}
\begin{enumerate}\item Let $\pi\simeq \chi \mathbf{1}_{\GL(2)} \rtimes \chi^{-1}$ be a representation of $\GSp_4(\Q_p)$ of type IIb in the notation of Table A.15 of \cite{NF}, where $\chi$ is an unramified unitary character of $\Q_p^\times$. Let $w_0$ be the (unique up to multiples) spherical function in the space of $\pi$. Then $B_{w_0}(1) \neq 0$ and for each pair of non-negative integers $\ell, m$, we have the bound
\[|B_{w_0}(h_p(\ell, m))| < (\ell +1)(2\ell + 2m+1) p^{-\ell - \frac{3m}2}|B_{w_0}(1)|.\]
Moreover, the space of $P_{1,p}$-fixed vectors in $\pi$ has a basis $\{w_0, w_1, w_2\}$ where $w_0$ is as above, and for $i=1,2$ we have $B_{w_i}(1) \neq 0$ and
\[|B_{w_i}(h_p(\ell, m))| = p^{-\frac{3}{2}\ell - \frac{3}{2}m}|B_{w_i}(1)|.\]
\item Let $\pi$ be isomorphic to one of the following  representations of $\GSp_4(\Q_p)$:
     \begin{itemize} \item The representation $\tau(T, \nu^{-1/2})$ of type VIb in the notation of Table A.15 of \cite{NF},
     \item The representation $L((\nu^{1/2}\xi\mathrm{St}_{\GL(2)}, \nu^{-1/2})$ of type Vb in the notation of Table A.15 of \cite{NF} where $\xi$ is the unramified non-trivial quadratic character of $\Q_p^\times$.
         \end{itemize}
                    Suppose that $\pi$ admits a $(\mathbf{1}, \theta)$ Bessel model and let $w$ be a non-zero vector in the  1-dimensional space of $P_{1,p}$-fixed vectors in $\pi$. Then $B_w(1) \neq 0$ and for each pair of non-negative integers $\ell, m$, we have
\[B_{w}(h_p(\ell, m)) = \pm p^{-2\ell - 2m}B_{w}(1).\] Above, the sign in $\pm$ equals -1 if and only if $\pi$ is of Type Vb and $\ell + m$ is odd.
\end{enumerate}
\end{prop}
\begin{proof}
Suppose that $\pi\simeq \chi \mathbf{1}_{\GL(2)} \rtimes \chi^{-1}$ is a representation of $\GSp_4(\Q_p)$ of type IIb where $\chi$ is an unramified unitary character of $\Q_p^\times$, and put $\alpha = \chi(\varpi)$ where $\varpi$ is a uniformizer in $\Q_p$. The fact that $B_{w_0}(1) \neq 0$ is a consequence of Sugano's formula \cite{sug}; see also Section 2 of \cite{pssmb}. We may therefore without loss of generality assume that $B_{w_0}(1) = 1$. Now by \cite[(12)]{pssmb}, we obtain that $|B_{w_0}(h_p(0, m))| < (2m+1) p^{-3m/2}$ (note that the Satake parameters $\alpha^\pm$ have absolute value 1) for all $m$. Now using \cite[(10)]{pssmb} and the above bound,
we obtain $$|B_{w_0}(h_p(\ell, m))| < \sum_{i=0}^\ell (2\ell + 2m - 2i +1)p^{-\frac{3\ell}{2} - \frac{3m}{2} + \frac{i}{2}} < (\ell+1)(2 \ell + 2m +1)p^{-\ell  - \frac{3m}{2}}.$$

Moreover, the space of $P_{1,p}$-fixed vectors in $\pi$ is 3-dimensional and has a basis $\{w_1, w_2, w_3\}$ consisting of eigenvectors for the commuting Hecke operators $T_{1,0}$ and $T_{0,1}$ with eigenvalues $(\alpha p^{3/2}, \alpha (p+1) p^{3/2})$, $(\alpha^{-1} p^{3/2}, \alpha^{-1} (p+1) p^{3/2})$, $(p^2, (\alpha + \alpha^{-1})p^{5/2})$  as calculated in Table 3 of \cite{pslongversion}. Since $w_1 + w_2 + w_3 = w_0$, it follows that  $\{w_0, w_1, w_2\}$ is a basis for the space of $P_{1,p}$-fixed vectors in $\pi$. Now by Proposition 6.1 of \cite{ps14} it follows that $B_{w_i}(h_p(\ell, m)) = \alpha^{\pm \ell \pm m} p^{-\frac{3 \ell}{2}-\frac{3 m}{2}}$ which completes the proof of the first part of the proposition.

Finally, let $\pi$ be either isomorphic to either $\tau(T, \nu^{-1/2})$ or $L((\nu^{1/2}\xi\mathrm{St}_{\GL(2)}, \nu^{-1/2})$. Using Table 3 of \cite{ps14}, we see that $w$ is an eigenvector for the operators $T_{1,0}$ and $T_{0,1}$ with pair of eigenvalues $(\eps p, \eps p(p+1))$ where $\eps =1$ for $\tau(T, \nu^{-1/2})$ and $\eps=-1$ for $L((\nu^{1/2}\xi\mathrm{St}_{\GL(2)}, \nu^{-1/2})$. Now by Proposition 6.1 of \cite{ps14} (noting that $m_0 =0$ in our case) it follows that $B_{w}(h_p(\ell, m)) = \eps^{-\ell-m}p^{-2\ell - 2m}B_{w}(1).$ The fact that $B_{w}(1) \neq 0$ follows from Theorem 8.2 of \cite{ps14}.
\end{proof}

The next proposition bounds the fundamental Fourier coefficients of cusp forms in the generalized Maass subspace.

\begin{prop}\label{p:fundSK}
Let $F \in S_k^*(\G_0^{(2)}(N))$ be a non-zero cusp form with Fourier coefficients $a(T)$. For fundamental $T \in \n_2$, we have
$$
| a(T) |
~\ll_{F, \epsilon} \begin{cases}
({\rm det}~T)^{\frac{k}{2} - \frac{7}{12}+\epsilon}
&\text{ unconditionally,} \\ ({\rm det}~T)^{\frac{k}{2} - \frac{3}{4}+\epsilon}
&\text{ assuming the generalized
Lindel\"of hypothesis.}\end{cases}
$$
 \end{prop}
\begin{proof}
We may assume without loss of generality that $F$ gives rise to an irreducible representation $\pi\simeq\pi_F$ (since any $F \in S_k^*(\G_0^{(2)}(N))$ may be written as a sum of such forms). By Proposition \ref{p:lfunctionsfacts} there exists a representation $\pi_0$ of $\GL_2(\A)$ such that \[L^N(s, \pi) = L^N(s, \pi_0) \zeta^N(s+1/2)  \zeta^N(s-1/2).\]

Let $T$ be fundamental and put $d = -4 \det(T)$. Then by Proposition 5.13 of \cite{JLS21+}, we have $$
| a(T) |
~\ll_{F}~
({\rm det}~T)^{\frac{k}{2} - \frac{3}{4}}
L(\tfrac12, \pi_0 \times \chi_d)^{1/2}.
$$

By \cite{PY19} we have $L(\tfrac12, \pi_0 \times \chi_d) \ll_{\pi_0, \eps} |d|^{1/3 + \eps}$ and conditionally on the generalized
Lindel\"of hypothesis we have $L(\tfrac12, \pi_0 \times \chi_d) \ll_{\pi_0, \eps} |d|^{\eps}$. The result follows.
\end{proof}

We can now prove the main theorem of this section.

\begin{thm}\label{t:SKRS}Let $k$ and $N$ be positive integers and let $F \in S_k^*(\G_0^{(2)}(N))$ be a non-zero cusp form with Fourier coefficients $a(T)$. Write $N=N_1N_2$ where $N_1$ is squarefree and $\gcd(N_1, N_2)=1.$ Then \[|a(T)| \ll_{F, \eps} L_T^{k-1+\eps} M_T^{k-3/2 +\eps} |d_T|^{\frac{k}{2} - \frac{7}{12} + \eps}\]
 for all $T \in \n_2$ satisfying \begin{equation}\label{e:skcondition}\gcd(L_TM_T, N_2) =1.\end{equation} Furthermore, if we assume the generalized
Lindel\"of hypothesis, then we have \[|a(T)| \ll_{F, \eps} L_T^{k-1+\eps} M_T^{k-3/2 +\eps} |d_T|^{\frac{k}{2} - \frac{3}{4} + \eps}\] for all $T \in \n_2$ satisfying \eqref{e:skcondition}.
\end{thm}
\begin{proof}We may assume without loss of generality that $F$ has the following properties:
\begin{itemize}
\item $F$ gives rise to an irreducible representation $\pi\simeq\otimes_v \pi_v$, and therefore (by Lemma \ref{l:equivconds}) $F$ is a Hecke eigenform at all good primes.
\item The adelization $\phi_F$ of $F$ corresponds to a factorizable vector $\phi_F = \otimes_v \phi_v$ in the space of $\pi$.

\item  If  $p|N_1$ and $\pi_p = \chi \mathbf{1}_{\GL(2)} \rtimes \chi^{-1}$ is of type IIb, then $\phi_p$ equals one of the vectors $w_0, w_1, w_2$ in the notation of Proposition \ref{p:SKlocal}.
\end{itemize}
Note that the set of $F$ satisfying the above properties linearly generate $S_k^*(\G_0^{(2)}(N))$. Note also (by Lemma \ref{l:SK properties} that if $p|N_1$ and $\pi_p$ is not of Type IIb, then $\pi_p$ must be of Type Vb or VIb and $\phi_p$ must correspond  to a (unique up to multiples) $P_{1,p}$-fixed vector in the space of $\pi_p$. Finally, Lemma \ref{l:SK properties} also tells us that at each prime $p \nmid N$, $\pi_p$ is of Type IIb.

Now, using essentially the same argument as in Corollary 2 of \cite{marzec21} (with $\Lambda=1$) together with Lemma \ref{l:SKkey}, we see that for $T$ as in the statement of the Theorem
\[a(T) \prod_{p|{L_TM_T}} B_{\phi_p}(1) = (L_TM_T)^k a(S_T) \prod_{p|{L_TM_T}} B_{\phi_p}(h_p(\ell_p, m_p)),\] where we write $L_T = \prod_p p^{\ell_p}$,  $M_T = \prod_p p^{m_p}$, and $S_T$ is any matrix of discriminant $d_T$. Above, we have used that $\gcd(L_TM_T, N_2) =1$ and consequently, $B_{\phi_p}(1) \neq 0$ for each $p|L_TM_T$ by the remarks at the beginning of this proof together with Proposition \ref{p:SKlocal}. Now, by Proposition \ref{p:SKlocal}, $\prod_{p|{L_TM_T}} |B_{\phi_p}(h_p(\ell_p, m_p))| \ll_\eps L_T^{-1+\eps}M_T^{-3/2+\eps} \prod_{p|{L_TM_T}} |B_{\phi_p}(1)|$. The result now follows from Proposition \ref{p:fundSK}.
\end{proof}
In the special case that $N$ is squarefree and $F \in S_k^*(\G_0^{(2)}(N))$ is a newform (in the sense of \cite[Section 3.2]{DPSS20}) we have an even sharper result.
\begin{prop}\label{t:SKRSnew}Let $N$ be a squarefree integer. Let $F \in S_k^*(\G_0^{(2)}(N))$ be a non-zero cusp form with Fourier coefficients $a(T)$, and suppose that $F$ is a newform (in the sense of \cite[Section 3.2]{DPSS20}). Then \[|a(T)| \ll_{F, \eps} \frac{L_T^{k-1+\eps} M_T^{k-3/2 +\eps}}{\gcd(L_T, N^\infty)(\gcd(M_T, N^\infty))^{1/2}} |d_T|^{\frac{k}{2} - \frac{7}{12} + \eps}\]
 for all $T \in \n_2$, and if we assume the generalized
Lindel\"of hypothesis, then the exponent of $|d_T|$ can be taken to be $\frac{k}{2} - \frac{3}{4} + \eps$.
\end{prop}
\begin{proof}The proof is essentially the same as Theorem \ref{t:SKRS} with the additional information that in this case, we know that $\pi_p$ is of Type Vb or Type VIb for all $p|N$ (see Lemma \ref{l:SK properties}. The sharper bound results from the fact that we are always in the situation of part (2) of Proposition \ref{p:SKlocal} for all $p|N$.
\end{proof}

\begin{proof} [Proof of Theorem \ref{DK-SK}] Note that if $T \in \n_2$ has the property that $\gcd(4\det(T), N)$ is squarefree, then we must have $\gcd(L_TM_T, N_2) =1$. Now the result follows from the observation that \[L_T^{k-1} M_T^{k-3/2} |d_T|^{\frac{k}{2} - \frac{7}{12}} = \frac{(4 \det(T))^{\frac{k-1}{2}}}{M_T^\frac12 |d_T|^\frac{1}{12}} \ll_k \det(T))^{\frac{k-1}{2}}.\]
\end{proof}

\subsection{Bounds for Fourier coefficients using Fourier--Jacobi expansions}\label{s:SKClass}
 The main result of this subsection is the following.
\begin{prop}\label{p:SKClass}
Let $k \ge 4$ be an even integer and
$N$ be an odd square-free positive integer.
Also let $f \in S_{2k-2}^{new}(\Gamma_0(N))$
be a newform and
$F \in S_k(\Gamma_0^{(2)}(N))$ be its
Saito-Kurokawa lift as described in \cite{AB15}.
Assume that $a(T)$ is the $T$-th Fourier coefficient of $F$.
Then, for any $\epsilon > 0$, we have
\begin{equation}\label{e:rssaitoclas}
|a(T)|
~\ll_{F, \epsilon}~
(\det T)^{\frac{k-1}{2}+\epsilon}.
\end{equation}
\end{prop}

 The statement of the above proposition refers to the work of Agarwal and Brown \cite{AB15} who, building upon work of Ibukiyama \cite{TI12}, wrote down a classical construction of Saito-Kurokawa lifts obtained from elliptic Hecke eigenforms of odd squarefree level $N$, and explicated some key properties of these lifts. In the next lemma, we show that the Saito-Kurokawa lifts constructed by Agarwal--Brown coincide with the newforms in $S_k^*(\G_0^{(2)}(N))$ in our sense.

\begin{lem}\label{l:SKclassautoequiv}Let $k \ge 2$ be even, let $N$ be an odd squarefree integer, and let $f \in S_{2k-2}^{new}(\Gamma_0(N))$ be a newform\footnote{Recall that a newform is a Hecke eigenform lying in the newspace.}; let $\pi_f$ be the automorphic representation of $\GL_2(\A)$ generated by $f$. Let $F \in S_k(\Gamma_0^{(2)}(N))$ be the
Saito-Kurokawa lift of $f$ as described in \cite{AB15}.

Then $F \in S_k^*(\G_0^{(2)}(N))$ and satisfies  $L^N(s, \pi_F) = L^N(s, \pi_f) \zeta^N(s+1/2)  \zeta^N(s-1/2)$. Furthermore, $F$ is a newform in the sense of \cite[Section 3.2]{DPSS20}.
Moreover, if $M$ is an integer and $G$ is a newform in $S_k^*(\G_0^{(2)}(M))$ satisfying $L^M(s, \pi_G) = L^M(s, \pi_f) \zeta^M(s+1/2)  \zeta^M(s-1/2)$ then $M=N$ and $G=cF$ for some non-zero complex number $c$.
\end{lem}
\begin{proof}Let $\pi \in \Pi(F)$. The fact that $L^N(s, \pi) = L^N(s, \pi_f) \zeta^N(s+1/2)  \zeta^N(s-1/2)$ follows from Theorem 3.5 of \cite{AB15} (observe that our normalization of the $L$-functions differ from theirs).  So by definition $F \in S_k^*(\G_0^{(2)}(N))$, and since $N$ is squarefree, Lemma \ref{l:SK properties} tells us that $\pi=\pi_F$ is irreducible and $F$ is a newform.
Next, let $G$ be as in the statement of the lemma. Since $G$ is a newform, Lemma \ref{l:SK properties} tells us that $M=N$.  Let $\pi'$ be the representation generated by the adelization of $G$. Applying Lemma \ref{l:SK properties} to $F$ and $G$, we see that the local components of $\pi$ and $\pi'$ are isomorphic at all finite primes, and hence at all places (since they are clearly isomorphic at infinity). It follows that the spaces of $\pi$ and $\pi'$ coincide. Since these representations have a 1-dimensional space of $K_{0,p}(N)$-fixed vectors at each prime $p$, it follows that the adelizations of $F$ and $G$ are multiples of each other. Hence $F$ and $G$ are multiples of each other.
\end{proof}

The rest of the subsection is devoted to the proof of Proposition \ref{p:SKClass}. We will assume throughout that $k \ge 4$ is an even integer and $N$ is an odd squarefree integer.
For any $F \in S_k(\G_0^{(2)}(N))$, the
Fourier--Jacobi expansion of $F$ is given by
$$
F(Z)
~:=~
\sum_{m =1}^\infty \phi_m(\tau, z) e^{2\pi im\tau'},
$$
where $Z = \begin{pmatrix} \tau & z \\ z & \tau' \end{pmatrix} \in \h_2$
with $\tau, \tau' \in\h$ and $z \in \C$. It is well-known \cite{EZ85, GS17}
that $\phi_m \in J_{k, m}^{cusp}(N)$, where $J_{k,m}^{cusp}(N)$ denotes
the space of Jacobi cusp form of weight $k$, index $m$ for the group
$\Gamma^J(N) := \Gamma_0(N) \ltimes \Z^2$. The following result of
Bringmann \cite[Theorem 3.41]{KB04} (see also Kohnen \cite{WK92, WK93}) on the
size of Fourier coefficients of a Jacobi cusp form is crucial for us.

\begin{prop}[{\cite[Theorem 3.41]{KB04}}]\label{WK92}
Let $\phi \in J_{k,m}^{cusp}(N)$ with Fourier coefficients $c(n,r)$. Then
$$
|c(n, r)|
~\ll_{\epsilon, k}~
\left( m^{3/4} + |D|^{1+\epsilon}\right)^\frac{1}{2}
\cdot
\frac{|D|^{k/2 - 1}}{m^\frac{k-1}{2}}
\cdot
\parallel \phi \parallel,
$$
where $\parallel \phi \parallel$ denotes the Petersson norm of $\phi$
which is defined by
$$
\parallel \phi \parallel^2
~=~
\langle \phi, \phi \rangle
:=
\frac{1}{[\G^J(1) : \G^J(N)]}
\int_{\Gamma^J(N) \backslash \h \times \C} | \phi(\tau, z)|^2 v^{k-3} e^{-4\pi my^2/v} dudvdxdy,
$$
with $\tau = u + iv \in \h$ and $z = x + iy \in \C$ and $D := r^2 - 4mn$.
\end{prop}

We need the following proposition whose proof is similar to
Kohnen and Sengupta \cite[Theorem 2]{KS17} (also see \cite[Section 6]{KP21})
 and relies on a result of Agarwal--Brown \cite{AB15} (see also \cite{MRV93, MR00, MR02}).

\begin{prop}
Let $k$ be a positive even integer and
$N$ be odd square-free positive integer.
Also let $f \in S_{2k-2}^{new}(\Gamma_0(N))$
be a newform and
$F \in S_k(\Gamma_0^{(2)}(N))$ be its
Saito-Kurokawa lift as described in \cite{AB15}.
Assume that $F$ has Fourier--Jacobi
coefficients $\{ \phi_m \}_{m \in \N}$.
Then, for any $\epsilon > 0$, we have
\begin{equation}\label{KS17}
\langle \phi_m, \phi_m \rangle
\ll_{F, \epsilon}
m^{k-1+\epsilon}.
\end{equation}
\end{prop}

\begin{proof}
For any positive integer $m$, \cite[Proposition 4.6 and p.659]{AB15}
implies
\begin{equation}\label{SK-level}
\frac{\langle \phi_m, \phi_m \rangle}{\langle \phi_1, \phi_1 \rangle}
 ~=
\sum_{d~|~m, \atop (d, N)=1} \alpha(d)d^{k-2}\lambda_f(m/d),
\end{equation}
where $\lambda_f(m)$ is the Hecke eigenvalue corresponding
to $f$ and hence by Deligne's bound we have
$| \lambda_f(m) | \le m^{k-3/2}\sigma_0(m)$.
Here $\sigma_0(m)$ denotes the number of
positive divisors of $m$. Further, $\alpha(d)$
is given by
$$
\alpha(d) := d\prod_{p|d}\left( 1+\frac{1}{p} \right) \le d \sigma_0(d).
$$
Therefore, one can easily see from \eqref{SK-level} that
$$
\frac{\langle \phi_{m}, \phi_{m} \rangle}{\langle \phi_1, \phi_1 \rangle}
 \le
m^{k-3/2}\sum_{d~|~m, \atop (d, N)=1} d^{1/2}\sigma_0(d)\sigma_0(m/d)
\ll_\epsilon
m^{k-1+\epsilon}\sum_{d~|~m, \atop (d, N)=1} 1
\ll_\epsilon
m^{k-1+\epsilon}.
$$
This completes the proof.
\end{proof}

\begin{proof}[Proof of Prop. \ref{p:SKClass}]The rest of the proof is similar to that of Kohnen \cite{WK92}, mutatis mutandis.
We sketch the proof here for convenience of reader.
Since both sides of \eqref{e:rssaitoclas} are invariant under $T \mapsto U^tTU$
for any $U \in {\rm GL}_2(\Z)$, we can assume that
$T = \begin{pmatrix} n & r/2\\ r/2 & m \end{pmatrix}$ is such that
$m = min~T$, where $min~T$ is the least positive integer
represented by $T$.
Now observe that $a(T)$ is the $(n,r)$-th Fourier coefficient of $\phi_m$.
Using \propref{WK92}, estimate \eqref{KS17} and $m = min~T ~\ll~ ({\rm det}~T)^{1/2}$
we complete the proof.
\end{proof}

\begin{rmk}
We would like to mention that if one could prove \eqref{KS17} for
$F \in S_k(\G_0^{(2)}(N))$ which is not of classical Saito--Kurokawa
type, that would lead to an improvement of the current best known
unconditional result \cite{WK92}.
\end{rmk}

\bibliography{FHS}{}
\bibliographystyle{alpha}

\end{document}